\documentclass[a4,12pt,dvipdfmx]{article}

\usepackage{amsmath,amssymb,amsthm}
\usepackage{docmute}
\usepackage{mathtools}
\usepackage{amsmath,amssymb,latexsym, amsthm,amscd,ascmac,graphicx}
\usepackage{float}
\usepackage{empheq}
\usepackage{amsthm,amsfonts,amssymb,amsmath,graphics,color,boxedminipage,amscd,colortbl,booktabs,graphicx}
\usepackage{fancybox}
\usepackage{float}
\usepackage{ascmac}
\usepackage{multirow}
\usepackage{framed}
\usepackage{comment}

\mathtoolsset{showonlyrefs=true}
\numberwithin{equation}{section}

\newcommand{\disp}{\displaystyle}
\newcommand{\Hei}{\mathbb{H}^d}
\newcommand{\bvec}[1]{\mbox{\boldmath $#1$}}

\newcommand{\G}{{\mathbb G}}
\newtheorem{thm}{Theorem}[section]

\newtheorem{prop}[thm]{Proposition}
\newtheorem{ttt}[thm]{Definition}
\newtheorem{exa}[thm]{Example}

\newtheorem{lem}[thm]{Lemma}
\newtheorem{Rem}[thm]{Remark}

\title{Results of existence and uniqueness for  the Cauchy problem of semilinear heat equations on stratified Lie groups}
\author{Hiroyuki HIRAYAMA  and  Yasuyuki OKA }
\date{}

\begin{document}
\maketitle


\begin{abstract}
The aim of this paper is to give existence and uniqueness results for solutions of  
the Cauchy problem for semilinear heat equations on stratified Lie groups $\G$ with the homogeneous dimension $N$.
We consider the nonlinear function behaves like $|u|^{\alpha}$ or 
$|u|^{\alpha-1}u$ $(\alpha>1)$ and the initial data $u_0$ belongs to the Sobolev spaces $L^p_s(\G)$ for $1<p<\infty$ and $0<s<N/p$.  Since stratified Lie groups $\G$ include the Euclidean space ${\mathbb R}^n$ as an example,  
our results  are an extension of the existence and uniqueness results obtained by F. Ribaud on ${\mathbb R}^n$ to $\G$. It should be noted that our proof is very different from it given by Ribaud on ${\mathbb R}^n$. We adopt the generalized fractional chain rule on $\G$ to obtain the estimate for the nonlinear term, which is very different from the paracomposition technique adopted by Ribaud on ${\mathbb R}^n$. By using the generalized fractional chain rule on $\G$, we can avoid the discussion of Fourier analysis on $\G$ and make the proof more simple. 

\end{abstract}

\section{Introduction and main results}
\noindent

On the Euclidean space ${\mathbb R}^n$, there are many results for the Cauchy problem for the semilinear heat equations
\begin{align}
\begin{cases}
\partial_tu(t,x)-\Delta u(t,x)=F(u(t,x)),\ (t,x)\in {\mathbb R}^+\times{\mathbb R}^n,\\
u(0,x)=u_0(x),\ x\in{\mathbb R}^n,
\end{cases}\label{1-1}
\end{align}
where $\Delta$ is a laplacian on ${\mathbb R}^n$ and $F$ is a nonlinear function which behaves like $|u|^{\alpha}$ or 
$|u|^{\alpha-1}u$ $(\alpha>1)$ (for instance, we refer to \cite{Fuji}, \cite{Giga}, \cite{Rib}, \cite{Wei} and reference therein). Particularly, in \cite{Rib}, F. Ribaud gave the existence and uniqueness of mild solutions to  \eqref{1-1} for $u_0\in L^p_s({\mathbb R}^n)$ with $1<p<\infty$, $s>0$, and $s\neq \disp{n}/{p}$. The mild solution means solution of the integral equation
\begin{align}
u(t,x)=e^{t\Delta}u_0(x)+\disp\int_0^te^{(t-\tau)\Delta}F(u(\tau,x))~d\tau   \label{1-2}
\end{align}
in $C([0,T);L^p_s({\mathbb R}^n))$ for some $T>0$, where $e^{t\Delta}u_0(x)=(u_0*E_t)(x)$, $E_t(x)$ is the heat kernel on ${\mathbb R}^n$. Since in the case of  $L^p_s({\mathbb R}^n)$ spaces, $0<s<n/p$, those spaces do not have the algebra property generally, F. Ribaud gave the estimate for nonlinear term on those spaces using Littlewood - Paley analysis (especially, the paracomposition technique) to solve \eqref{1-2} (see \cite{Rib}, Theorem 1.2 and Theorem 1.3).

On the other hand, on stratified Lie groups ${\mathbb G}$ with the homogeneous dimension $N$ including ${\mathbb R}^n$ (1 step stratified Lie group) and the Heisenberg group (2 step stratified Lie group) as an example, in the last decade, one consider the Cauchy problem for the semilinear heat equations
\begin{align}
\begin{cases}
\partial_tu(t,x)+{\mathcal L} u(t,x)=F(u(t,x)),\ (t,x)\in {\mathbb R}^+\times{\mathbb G},\\
u(0,x)=u_0(x),\ x\in{\mathbb G},
\end{cases}\label{1-3}
\end{align}
where ${\mathcal L}$ is a sublaplacian on ${\mathbb G}$ (for instance, we refer to \cite{Bru},  \cite{Oka10}, \cite{Oka13}, \cite{Poh}, \cite{Ruz2} and reference therein). Particularly, in \cite{Bru}, T. Bruno {\it et al.} consider that $F:{\mathbb R}\rightarrow{\mathbb R}$ is smooth satisfying $D^hF(0)=0$, whenever $h\leq [s]+1$. They gave the existence and uniqueness of mild solutions to \eqref{1-3} for $u_0\in L^p_s({\mathbb G})\cap L^{\infty}({\mathbb G})$ with $1<p<\infty$ and $s\geq 0$, where the spaces $L^p_s({\mathbb G})$ are fractional Sobolev spaces on ${\mathbb G}$ (we give the definition of the space $L^p_s(\G)$ below).   
The mild solution means solution of the integral equation
\begin{align}
u(t,x)=e^{-t{\mathcal L}}u_0(x)+\disp\int_0^te^{-(t-\tau){\mathcal L}}F(u(\tau,x))~d\tau \label{1-4}
\end{align}
in $C([0,T);L^p_s(\G))$ for some $T>0$, where $e^{-t{\mathcal L}}u_0(x)=(u_0*h_t)(x)$, $h_t(x)$ is the heat kernel on ${\mathbb G}$. 

Bruno {\it et al.} use the algebra property of the space $L^p_s({\mathbb G})\cap L^{\infty}({\mathbb G})$, $1<p<\infty$ and $s\geq 0$ (see \cite{Bru}, Theorem 1.2 and Theorem 6.3). By Proposition \ref{embedding} (Sobolev embedding theorem) below, for $1<p<\infty$ and $s>N/p$, the space $L^p_s(\G)\cap L^{\infty}(\G)$  is equal to $L^p_s(\G)$. Hence,  in the case of the space $L^p_s(\G)$, $s>N/p$, our theorem \ref{th1} below is already proved in \cite{Bru}.
Since, in the case $0<s<N/p$, we can not use the algebra property of $L^p_s(\G)$ generally, the aim of this paper is to improve the conditions for the initial value by considering the nonlinear term like as a power type. Thus we will give the existence and uniqueness results for solutions of \eqref{1-4} when $u_0$ belongs to $L^p_s(\G)$, $1<p<\infty$ and $0<s<N/p$. Namely, we obtain the same results in \cite{Bru} without $L^{\infty}$ condition. This means the extending Ribaud's result on ${\mathbb R}^n$ to $\G$. To do this, for $s$ and $F$, we assume that 
\begin{align}
s> \disp{N}/{p}-{N}/{\alpha}\ \ \text{and}\ \ s\geq 0\label{1-5}
\end{align}
 and that \\
 N1) there exists $\alpha>1$ such that
 \begin{enumerate}
 \item $s\leq \disp{N(\alpha-1)}/{(p\alpha)}$,
 \item $F$: ${\mathbb R}\rightarrow{\mathbb R}$ satisfies $F(0)=0$ and there exists a constant $C>0$ such that 
 \begin{align}
 |F(x)-F(y)|\leq C|x-y|(|x|^{\alpha-1}+|y|^{\alpha-1})
 \end{align}
 \end{enumerate}
 or,\\
N2) there exists $\alpha>1$ such that
\begin{enumerate}
\item $s$ satisfies $0<s<N/p$ and
\begin{align}
\disp\frac{N(\alpha-1)}{p\alpha}<s<\left(\frac{N}{p}+1\right)\disp\frac{\alpha-1}{\alpha}, \label{apapap}
\end{align}
\item $F$: ${\mathbb R}\rightarrow{\mathbb R}$ is $[\alpha]$ times differentiable, $F^{(j)}(0)=0$ for $j=0,\cdots, [\alpha]-1$, $F^{([\alpha])}(0)=0$ if $\alpha\not\in {\mathbb N}$, and 
\begin{align}
|F^{([\alpha])}(x)-F^{([\alpha])}(y)|\leq C|x-y|^{\alpha-[\alpha]}.
\end{align}
\end{enumerate}
\begin{Rem}\label{12_8}\rm
(i) By means of  Proposition \ref{embedding} (Sobolev embedding theorem on $\G$) below, if $u\in C([0,T); L^p_s(\G))$ with $s$ as in \eqref{1-5}, then $u\in C([0,T);L^{\tilde{p}}(\G))$, where $1/\tilde{p}=1/p-s/N$. 
We note that $\tilde{p}>\alpha$ holds. Hence the term $F(u)$ in \eqref{1-4} is well defined in ${\mathcal D}^{\prime}((0,T)\times\G)$. 

(ii) Let $\alpha>1$. We can easily derive the following properties from 
the condition N2)$-$(ii); there exists a constant $C>0$ such that 
\begin{align}
|F^{(j)}(x)-F^{(j)}(y)|\leq C|x-y|(|x|^{\alpha-j-1}+|y|^{\alpha-j-1})
\end{align}
for all $j=0,\cdots,[\alpha]-1$ and 
\begin{align}
|F^{(j)}(x)|\leq C|x|^{\alpha-j}
\end{align}
for all $j=0,\cdots, [\alpha]$.
\end{Rem}

We address to state our main results as follows.

\begin{thm}\label{th1}
Let $1<p<\infty$ and $s>s_c={N}/{p}-{2}/{(\alpha-1)}$. Assume that \eqref{1-5} holds, and that N1) or N2) is fulfilled. Then for any initial value $u_0$ in $L^p_s(\G)$, there exist $T_0>0$ and a unique solution $u(t,x)$ of \eqref{1-4} in $C([0,T_0);L^p_s(\G))$. 
\end{thm}
\begin{Rem}\rm
(i) Since ${\mathcal L}$ is $\delta_{\lambda}$ homogeneous of degree two, if we consider the nonlinear heat equations
\begin{align}
\partial_tu(t,x)+{\mathcal L}u(t,x)=u(t,x)|u(t,x)|^{\alpha-1},\  t>0,\ x\in \G,  \label{4_8_1}
\end{align}
then for a solution $u(t,x)$ of \eqref{4_8_1} and for each $\lambda>0$, the functions $u_{\lambda}$ defined by 
\[u_{\lambda}(t,x)=\lambda^{2/(\alpha-1)}u(\lambda^2t, \delta_{\lambda}(x))\]
are also solutions, where $\delta_{\lambda}$ is a dilation on $\G$, which will be defined below, and one can check that $u(0)$ and $u_{\lambda}(0)$ have the same norm in $L^{\infty}({\mathbb R}^+,\dot{L}^p_s(\G))$ if and only if 
\[s=s_c=\disp\frac{N}{p}-\disp\frac{2}{\alpha-1}. \]
{\color{black}{(ii) Let $\alpha>1$. Assume $p\in (1,\infty)$ satisfies $p\geq \alpha$ and $p>N(\alpha-1)/2$. The condition $p>N(\alpha-1)/2$ is equivalent to $s_c<0$.  
Then Theorem \ref{th1} means that for any initial value $u_0$ in $L^p(\G)$, there exist $T_0>0$ and a unique solution $u(t,x)$ of \eqref{1-4} in $C([0,T_0);L^p(\G))$. }}\\
 {\color{black}{(iii) We define $T_{\mathrm{max}}$ by 
 \[T_{\mathrm{max}}:= \sup\{~T>0\ |\ \text{there\ exists}\ u\in C([0,T);L^p_s(\G))\ \text{solution\ of}\  \eqref{1-4}\}\]
 and assume $T_{\mathrm{max}}<\infty$. Then by \eqref{6-6} in the proof of Theorem \ref{th1} below, we can see that there exists a positive constant $C$ (independent of $u_0$) such that for any $t\in [0, T_{\mathrm{max}})$, 
\begin{align}
 T_{\mathrm{max}}-t\geq C\|u(t,\cdot)\|_{L^p_s}^{-\frac{2}{s-s_c}}, \label{5_12_1}.
 \end{align}
We note that if $s=0$, then $\disp{2}/{(s-s_c)}=\disp\left({1}/{(\alpha-1)}-{N}/{2p}\right)^{-1}$.
By the estimate \eqref{5_12_1}, we can see that
 \[\disp\lim_{t\rightarrow T_{\mathrm{max}}-0}\|u(t,\cdot)\|_{L^p_s}=\infty.\]  
}} 

\end{Rem}

To prove Theorem \ref{th1}, we need the following theorem. 

\begin{thm}\label{th2}
Let $1<p<\infty$ and 
\[N/p-N/{\alpha}<s<\disp{N}/{p},\  s\geq 0.\]
We put 
\begin{align}
s_{\alpha}=s-(\alpha-1)\disp\left({N}/{p}-s\right).\label{s_alpha}
\end{align}
If N1) or N2) is fulfilled, then for any $u\in L^p_s(\G)$, $F(u)$ belongs to $L^p_{s_{\alpha}}(\G)$. 
Furthermore, there exists a constant $C>0$ independent of $u$ such that 
\begin{align}
\|F(u)\|_{L^p_{s_{\alpha}}}\leq C\|u\|_{L^p_s}^{\alpha}.\label{2121}
\end{align}
\end{thm}
\noindent
\begin{Rem}\rm
(i) Since the algebra property does not hold generally in $L^p_s(\G)$ for $s<N/p$, to prove Theorem \ref{th1},  we need  Theorem \ref{th2}. By using Theorem \ref{th2}, we can control the loss of smoothness on the $L^p_s(\G)$ scale coming from the composition by $F$. 

(ii) The condition N1)$-$(i) is equivalent to
\[s_{\alpha}\leq 0\]
and \eqref{apapap} in N2)$-$(i) is equivalent to
\[0<s_{\alpha}<\alpha-1.\]
The restriction $s_{\alpha}<\alpha-1$ is due to the lack of smoothness of the nonlinear term $F$ at $x=0$. 
(in detail, we refer to \cite{Rib}). 
\end{Rem}
In \cite{Rib}, F. Ribaud adopted the paracomposition technique to prove Theorem \ref{th2} on ${\mathbb R}^n$, but its technique, based on the Fourier transform, does not fit the case of stratified Lie groups with higher step $r>2$. So, we will prove Theorem \ref{th2} by using Lemma \ref{3_11_2} (the generalized fractional chain rule on $\G$) below, which has been proved on ${\mathbb R}^n$ by T. Kato (see \cite{Kato}, Lemma A3 and see also \cite{Chr}, Proposition 3.1).

The content of the paper is as follows. In section 2, we introduce the definition and properties of the stratified Lie groups $\G$. We also introduce the properties of the heat kernel on $\G$ and give  Proposition \ref{17244} ($L^{\alpha} - L^{\beta}$ estimate.  In section 3, we summarize the Sobolev spaces $L^p_s$ on $\G$. Especially, we introduce Proposition \ref{embedding} (Sobolev embedding theorem) and so on and give Proposition \ref{L4} ($L^p_{s+\theta} - L^p_s$ estimate). In section 4, let us assume that Theorem \ref{th2} holds. Then we prove Theorem \ref{th1}.   In section 5, we prove Theorem \ref{th2} utilizing the generalized fractional chain rule on $\G$. Finally, in section 6, we summarize the needed lemmas to prove Theorem \ref{th2}. Especially, we give Lemma \ref{3_11_2} (the generalized fractional chain rule on $\G$).

Throughout this paper, the letters $C$ and so on will be used to denote positive constants, which are independent of the
main variables involved and whose values may vary at every occurrence.
By writing $f\lesssim g$, we mean $f\leq Cg$ for some positive constant $C>0$. The notation $f\sim g$ will stand for $f\lesssim g$ and $g\lesssim f$.

\section{The stratified Lie groups}
\subsection{Definition and properties}
\noindent

We recall the definition of the stratified Lie groups (see \cite{Ugu}, \cite{Fis}, \cite{Folland2}, \cite{Fol}, \cite{Ruz}, \cite{Var2} and reference therein).  In particular, we adopt the definition introduced in \cite{Ugu}. A Lie group $({\mathbb R}^d,\cdot~)$ is called a stratified Lie group (a homogeneous Carnot group) and is denoted by ${\mathbb G}=({\mathbb R}^d,\ \cdot\ ,\delta_{\lambda})$ if it satisfies the following two conditions:
\begin{enumerate}
\item For $d=d_1+\cdots+d_r$, $d_1,\cdots,d_r\in {\mathbb N}$ and $r\geq 1$, the decomposition ${\mathbb R}^d={\mathbb R}^{d_1}\times\cdots\times{\mathbb R}^{d_r}$ holds and for any $\lambda>0$, the dilation $\delta_{\lambda}$ : ${\mathbb R}^d\rightarrow{\mathbb R}^d$ defined by
\[\delta_{\lambda}(x)=\delta_{\lambda}(x^{(1)},\cdots,x^{(r)}):=(\lambda x^{(1)},\lambda^2 x^{(2)},\cdots, \lambda^rx^{(r)}),\]
where $x^{(k)}\in{\mathbb R}^{d_k}$ for $k=1,\cdots, r$, is an automorphism of the group $({\mathbb R}^d,\cdot)$.  
\item Let $d_1$ be as in (i) and $X_1,\cdots,X_{d_1}$ be the left invariant vector fields on $({\mathbb R}^d,\cdot)$ such that $X_k(0)=\disp\left.\frac{\partial}{\partial x_k}\right|_{x=0}$ for $k=1,\cdots, d_1$. Then H${\rm\ddot{o}}$rmander condition
\[\mathrm{rank}(\mathrm{Lie}\{X_1,\cdots,X_{d_1}\}(x))=d\]
holds for any $x\in{\mathbb R}^d$, that is, the iterated commutators of $X_1,\cdots,X_{d_1}$ span the Lie algebra ${\mathfrak g}$ of $({\mathbb R}^d,\cdot)$. 
\end{enumerate}
As a remark, the definition of the stratified Lie groups implies that  the Lie algebra ${\mathfrak g}$ of ${\mathbb G}$ admits a stratification, {\it i.e.} a direct sum decomposition 
\[ {\mathfrak g}=V_1\oplus V_2\oplus\cdots\oplus V_r\]
such that $[V_1,V_{i-1}]=V_i$ if $2\leq i\leq r$ and $[V_1,V_r]=\{0\}$ and 
it is known that the composition law $\cdot$ of the stratified Lie groups takes the following form (for instance, see \cite{Ugu}).
\begin{prop}[\cite{Ugu}]\label{K80}
Let ${\mathbb G}=({\mathbb R}^d,\ \cdot\ ,\delta_{\lambda})$ be the stratified Lie group. The  composition law $\G\times\G\ni (x,y)\mapsto x\cdot y\in\G$ has polynomial component functions. Moreover we have for $x,y\in{\mathbb R}^d$, 
\[(x\cdot y)_1=x_1+y_1,\ (x\cdot y)_j=x_j+y_j+Q_j(x,y),\ 2\leq j\leq d,\]
and the following facts hold:
\begin{enumerate}
\item[\rm{(i)}] $Q_j(x,y)$ only depends on $x_1,\cdots,x_{j-1}$ and $y_1,\cdots,y_{j-1}$,
\item[\rm{(ii)}] $Q_j(x,y)$ is a sum of mixed monomials in $x,y$,
\item[\rm{(iii)}] there exists $\sigma_j>0$ such that $Q_j(\delta_{\lambda}x,\delta_{\lambda}y)=\lambda^{\sigma_j}Q_j(x,y)$.
\end{enumerate}
More precisely, $Q_j(x,y)$ only depends on the $x_k$'s and $y_k$'s with $\sigma_k<\sigma_j$.
\end{prop} 
\noindent
The number $r$ is called the step of ${\mathbb G}$ and the number
\[N=\disp\sum_{k=1}^rkd_k\]
is called the homogeneous dimension of ${\mathbb G}$, respectively.  
 Let $X_1,\cdots, X_d$ be the Jacobian basis of ${\mathfrak g}$, i.e.
\[X_j(0)=\disp\left.\frac{\partial}{\partial x_j}\right|_{x=0}\ \text{for}\  j=1,\cdots, d\]
(for details, see \cite{Ugu}).
We can also denote the Jacobian basis $\{X_1,\cdots,X_d\}$ by
\[X_1^{(1)},\cdots,X_{d_1}^{(1)},\cdots,X_1^{(r)},\cdots,X_{d_r}^{(r)}.\]
Then $X_k^{(i)}$, $1\leq i\leq r$, takes the form
\[X_k^{(i)}=\disp\frac{\partial}{\partial x_k^{(i)}}+\disp\sum_{l=i+1}^r\disp\sum_{m=1}^{d_l}a_{k,m}^{(i,l)}(x^{(1)},\cdots,x^{(l-i)})\disp\frac{\partial}{\partial x_m^{(l)}},\ 1\leq k\leq d_i,\]
where $a_{k,m}^{(i,l)}$ is a $\delta_{\lambda}$-homogeneous polynomial function of degree $l-i$ (see \cite{Ugu}, \cite{Ruz} and \cite{Var2} for more details), namely, $a_{k,m}^{(i,j)}(\delta_{\lambda}(x))=\lambda^{l-i}a_{k,m}^{(i,j)}(x)$. Especially, for $i=1$, we have
\begin{align*}
X_k^{(1)}=\disp\frac{\partial}{\partial x_k^{(1)}}+\disp\sum_{l=2}^r\disp\sum_{m=1}^{d_l}a_{k,m}^{(1,l)}(x^{(1)},\cdots,x^{(l-1)})\disp\frac{\partial}{\partial x_m^{(l)}},\ 1\leq k\leq d_1.
\end{align*}
We denote by $X_k$ for $1\leq k\leq d_1$, the left invariant vector fields $X_k^{(1)}$ in the first stratum of ${\mathbb G}$. The left invariant vector fields $X_1,\cdots, X_{d_1}$ are called the (Jacobian) generators of ${\mathbb G}$. A sublaplacian of ${\mathbb G}$ is denoted by ${\mathcal L}=-\disp\sum_{i=1}^{d_1}X_i^2$. The operator ${\mathcal L}$ is not elliptic but hypoelliptic. ${\mathcal L}$ is $\delta_{\lambda}$ homogeneous of degree two, that is, for any $\lambda>0$, we have 
\[{\mathcal L}(u(\delta_{\lambda}(x)))=\lambda^2({\mathcal L}u)(\delta_{\lambda}(x))\]
for any $x\in\G$. Since the system $\bold{X}=\{X_1,\cdots,X_{d_1}\}$ satisfies the H${\rm\ddot{o}}$rmander condition,  the Carnot-Carath${\rm \acute{e}}$odory distance $\rho_{\bold{X}}(x,x^{\prime})$ can be defined (see \cite{Ugu} and \cite{Var2}  for details). We denote by $\rho(x)$ the distance from the origin, {\it i.e}. $\rho(x)=\rho_{\bold{X}}(\bvec{e},x)$, where $\bvec{e}$ is an identity element of ${\mathbb G}$.  The homogeneous of degree of  $\rho$  is one, that is, 
\begin{align*}
\rho(\delta_{\lambda}(x))=\lambda \rho(x),\ x\in {\mathbb G}
\end{align*}
for any $\lambda>0$ (see \cite{Ugu}, Proposition 5.2.6).
The following estimate also holds:
\begin{align*}
 \rho({x^{\prime}}^{-1}\cdot x)\leq \rho(x)+\rho({x^{\prime}}).
\end{align*}
Define
\begin{align*}
|x|_{\mathbb G}=\disp\left(\disp\sum_{j=1}^r|x^{(j)}|^{\frac{2r!}{j}}\right)^{\frac{1}{2r!}}
\end{align*}
for $x=(x^{(1)},\dots,x^{(r)})\in {\mathbb G}$. Then there exists a constant $C>0$ such that 
\[C^{-1}|x|_\G\leq \rho(x)\leq C|x|_\G\]
for any $x\in {\mathbb G}$ (see \cite{Ugu}, Proposition 5.1.4 and Theorem 5.2.8). Since the homogeneous of degree of $|\cdot|_{\mathbb G}$ is also one,  we deal with $|\cdot|_{\mathbb G}$ as a homogeneous norm on ${\mathbb G}$. 

The stratified Lie groups $\G$ are locally compact Hausdorff spaces and their  Haar measure of  ${\mathbb G}$ is the Lebesgue measure $dx=dx_1dx_2\cdots dx_d$.
We can see that the following relation holds;
\[\disp\int_{{\mathbb G}}f(\delta_{\lambda}(x))dx=\lambda^{-N}\disp\int_{{\mathbb G}}f(x)dx.\]
Let $f$ and $h$ be suitable functions. Then the convolution $f*h$ of $f$ with $h$ on ${\mathbb G}$ is defined by
\[(f*h)(x)=\disp\int_{\mathbb G}f(x^{\prime})h({x^{\prime}}^{-1}\cdot x)dx^{\prime}=\disp\int_{{\mathbb G}}f(x\cdot {x^{\prime}}^{-1})h(x^{\prime})dx^{\prime}.\] 
The convolution $*$ is non commutative. The relationship between the left invariant vector fields $X_i$ and the convolution $*$ is $X_i(f*h)(x)=(f* X_ih)(x)$.

We set 
\[L^p(\G)=\disp\left\{f\ |\ \|f\|_p<\infty\right\}\]
with a norm $\|\cdot\|_p$ defined by 
\[\|f\|_p=\left(\disp\int_{\G}|f(x)|^pdx\right)^{\frac{1}{p}}\]
for $1\leq p<\infty$ and if $p=\infty$, we set 
\[L^{\infty}(\G)=\{f\ |\ \|f\|_{\infty}<\infty\}\]
with a norm $\|\cdot\|_{\infty}$ defined by 
\[\|f\|_{\infty}=\mathrm{inf}\disp\left\{M\ |\ \text{the measure of}\ \{x\in\G\ |\ |f(x)|>M\}\ \text{is}\  0\right\}.\]
Let $I=(i_1,\cdots,i_{\beta})\in \{1,\cdots,d_1\}^{\beta}$, $\beta\in{\mathbb N}_0$. Then we put $|I|=\beta$ and $X^I=X_{i_1}\cdots X_{i_\beta}$, with the convention that $X^I=\mathrm{Id}$ if $\beta=0$. Now we denote by $S(\G)$
the class of all functions $f\in C^{\infty}(\G)$ such that 
\[\|f\|_{\kappa}=\disp\sup_{|I|\leq \kappa\atop x\in\G}(1+\rho(x))^{\kappa}|X^If(x)|<\infty.\]  

We recall a polar coordinate transform on $\G$ (in detail, see \cite{Fol}). In this paper, 
we say $u$ is radial if there exists $v:[0,\infty)\rightarrow {\mathbb R}$ such that $u(x)=v(|x|_{\G})$. Let ${\mathbb S}=\{x\in\G:|x|_{\G}=1\}$. Then there exists a unique Radon measure $\mu$ on ${\mathbb S}$ such that for any $u\in L^1(\G)$, 
\[\disp\int_{\G}u(x)~dx=\disp\int_0^{\infty}\int_{{\mathbb S}}u(ry)r^{N-1}~d\mu(y)dr\]
(see \cite{Fol}, Proposition (1.15)). Especially, if $u$ is radial, then we have
\[\disp\int_{\G}u(x)~dx=\omega_N\disp\int_0^{\infty}v(r)r^{N-1}dr,\]
where $\omega_N=\int_{{\mathbb S}}~d\mu(y)$.

\begin{exa}[1-dimensional Heisenberg group ${\mathbb H}=({\mathbb R}^3,\ \cdot\ ,\delta_{\lambda} )$]\rm 

The group law of ${\mathbb H}$ is given by
\[(z,s)\cdot (z^{\prime},s^{\prime})=\disp\left(z+z^{\prime},s+s^{\prime}+\disp\frac{1}{2}(yx^{\prime}-xy^{\prime})\right),\]
where, $z=(x,y),z^{\prime}=(x^{\prime},y^{\prime})\in{\mathbb R}^{2}$, $s\in{\mathbb R}$. The dilation $\delta_{\lambda}$ is
\[\delta_{\lambda}(x,y,s)=(\lambda x, \lambda y, \lambda^2 s).\]
Moreover the left invariant vector fields are given by
\begin{align*}
X = \disp\frac{\partial}{\partial x}+\frac{1}{2}y\frac{\partial}{\partial s},\ 
Y = \disp\frac{\partial}{\partial y}-\frac{1}{2}x\frac{\partial}{\partial s},\ 
S= \disp\frac{\partial}{\partial s}
\end{align*}
and it holds that
\[[X,Y]=-S,\ [X,X]=[Y,Y]=[X,S]=[Y,S]=[S,S]=0.\]
Thus, the step of the Heisenberg group ${\mathbb H}$ is $2$ and the homogeneous dimension of  the Heisenberg group ${\mathbb H}$ is $4$. The sublaplacian ${\mathcal L}$ is 
\[{\mathcal L}=-X^2-Y^2.\]
\end{exa}

\begin{exa}[A group of H type ${\mathbb H}_2^2=({\mathbb R}^6,\ \cdot\ ,\delta_{\lambda})$]\rm
The group law of ${\mathbb H}^2_2$ is given by
\[(z,s)\cdot (z^{\prime},s^{\prime})=
\begin{pmatrix}
  z+z^{\prime}\\
  s_1+s_1^{\prime}+\disp\frac{1}{2}\left(-x_1x_2^{\prime}+x_2x_1^{\prime}-y_1y_2^{\prime}+y_2y_1^{\prime}\right)\\\\
  s_2+s_2^{\prime}+\disp\frac{1}{2}\left(x_1y_1^{\prime}-x_2y_2^{\prime}-x_1^{\prime}y_1+y_2x_2^{\prime}\right)
\end{pmatrix}
,\]
where, $z=(x_1,x_2,y_1,y_2)$, $z^{\prime}=(x_1^{\prime},x_2^{\prime},y_1^{\prime},y_2^{\prime})\in{\mathbb R}^{4}$, $s=(s_1,s_2), s^{\prime}=(s_1^{\prime},s_2^{\prime})\in{\mathbb R}^2$.  
The dilation $\delta_{\lambda}(x_1,x_2,y_1,y_2,s_1,s_2)$ is 
\[\delta_{\lambda}(x_1,x_2,y_1,y_2,s_1,s_2)=(\lambda x_1, \lambda x_2, \lambda y_1, \lambda y_2, \lambda^2 s_1, \lambda^2s_2).\]
Moreover the left invariant vector fields are given by
\begin{align*}
X_1 &= \disp\frac{\partial}{\partial x_1}+\frac{1}{2}\left(x_2\frac{\partial}{\partial s_1}-y_1\frac{\partial }{\partial s_2}\right),\  
X_2 = \disp\frac{\partial}{\partial x_2}+\frac{1}{2}\left(-x_1\frac{\partial}{\partial s_1}+y_2\frac{\partial }{\partial s_2}\right),\\ 
Y_1 &= \disp\frac{\partial}{\partial y_1}-\frac{1}{2}\left(y_2\frac{\partial}{\partial s_1}+x_1\frac{\partial }{\partial s_2}\right),\ 
Y_2 = \disp\frac{\partial}{\partial y_2}-\frac{1}{2}\left(-y_1\frac{\partial}{\partial s_1}-x_2\frac{\partial }{\partial s_2}\right),\\
S_1 &= \disp\frac{\partial}{\partial s_1},\ S_2 = \disp\frac{\partial}{\partial s_2}
\end{align*}
and it holds that
\[[X_1,Y_1]=-S_1,\ [X_2,Y_2]=-S_2\]
and others are $0$. Thus, the step of ${\mathbb H}_2^2$ is $2$ and the homogeneous dimension of  ${\mathbb H}_2^2$ is $8$. The sublaplacian ${\mathcal L}$ is 
\[{\mathcal L}=-X_1^2-X^2_2-Y_1^2-Y_2^2.\]
\end{exa}

\begin{exa}[${\mathbb I}=({\mathbb R}^4,\ \cdot\ ,\delta_{\lambda} )$]\rm The group low of ${\mathbb I}$ is given by
\[(x_1,x_2,x_3,x_4)\cdot (y_1,y_2,y_3,y_4)=
\begin{pmatrix}
  x_1+y_1\\
  x_2+y_2\\
x_3+y_3+1/2(y_2x_1-y_1x_2)\\
x_4+y_4+1/2(y_3x_1-y_1x_3)+1/12(x_1-y_1)(y_2x_1-y_1x_2)
\end{pmatrix}.
\]
The dilation $\delta_{\lambda}(x_1,x_2,x_3,x_4)$ is 
\[\delta_{\lambda}(x_1,x_2,x_3,x_4)=(\lambda x_1, \lambda x_2, \lambda^2 x_3, \lambda^3 x_4).\]
Moreover the left invariant vector fields are given by
\begin{align*}
X_1 &=\disp\frac{\partial }{\partial x_1}-\disp\frac{1}{2}x_2\frac{\partial}{\partial x_3}-\disp\left(\frac{1}{2}x_3+\frac{1}{12}x_1x_2\right)\frac{\partial }{\partial x_4},\ 
X_2 =\disp\frac{\partial}{\partial x_2}+\disp\frac{1}{2}x_1\frac{\partial}{\partial x_3}+\disp\frac{1}{12}x_1^2\disp\frac{\partial}{\partial x_4},\\
X_3 &=\disp\frac{\partial}{\partial x_3}+\disp\frac{1}{2}x_1\frac{\partial}{\partial x_4},\ X_4 =\disp\frac{\partial}{\partial x_4}
\end{align*}
and 
it holds that
\[[X_1,X_2]=X_3,\ [X_1,[X_1,X_2]]=X_4\]
and others are $0$. Thus, the step of ${\mathbb I}$ is $3$ and the homogeneous dimension of  ${\mathbb I}$ is $7$. The sublaplacian ${\mathcal L}$ is 
\[{\mathcal L}=-X_1^2-X^2_2.\]
\end{exa}

\subsection{The heat kernel}
\noindent

We can construct a heat semigroup $\{e^{-t{\mathcal L}}\}_{t>0}$ of linear operators on $L^1(\G)+L^{\infty}(\G)$ with the infinitesimal generator $-{\mathcal L}$ by a theorem of G. We denote by ${\mathcal D}={\mathcal D}(\G)$ the space of smooth functions with compact support. A. Hunt in \cite{Hunt}, whose properties are summarized as follows. 

\begin{prop}[\cite{Fis}, \cite{Folland2}, \cite{Hunt}]\label{3-1}
\begin{enumerate}
\item[\rm(i)] There exists a unique semigroup $\{e^{-t{\mathcal L}}\}_{t>0}$ of linear operators on $L^1(\G)+L^{\infty}(\G)$ satisfying conditions:
\begin{enumerate}
\item[\rm(a)] $e^{-t{\mathcal L}}f(x)=(f*h_t)(x)$, where $h_t(x)=h(t,x)$ is $C^{\infty} ((0,\infty)\times \G)$. Moreover,
$h_t\in{\mathcal S}(\G)$.
\item[\rm(b)] $h_t(x)\geq 0$, $\int_{{\mathbb G}}h_t(x)~dx=1$, $h_t(x)=h_t(x^{-1})$, $({\partial}/{\partial t}+{\mathcal L})h_t=0$ and $h_{r^2t}(\delta_r(x))=r^{-N}h_t(x)$, $r>0$, $x\in{\mathbb G}$, where $N$ is the homogeneous dimension of $\G$.
\item[\rm(c)] $\{e^{-t{\mathcal L}}\}_{t>0}$ is a contraction semigroup and strongly continuous on $L^p(\G)$ for $1<p<\infty$. 
Furthermore, for any $f\in{\mathcal D}(\G)$ and any $p\in [1,\infty]$, we have the convergence 
\[\disp\left\|\disp\frac{1}{h}(e^{-h{\mathcal L}}f-f)+{\mathcal L}f\right\|_p \rightarrow 0\]
as $h\rightarrow 0$.
\end{enumerate}
\item[\rm(ii)] Let ${\mathcal L}_p$ be minus the infinitesimal generator of $\{e^{-t{\mathcal L}}\}_{t>0}$ on $L^p(\G)$. Then 
\begin{enumerate}
\item[\rm(a)] For $1<p<\infty$, ${\mathcal L}_p$ is a closed operator on $L^p(\G)$ whose domain is dense in $L^p(\G)$.
\item[\rm(b)] For $1<p<\infty$, ${\mathcal D}\subset \mathrm{Dom}~({\mathcal L}_p)$. Furthermore ${\mathcal L}_pu={\mathcal L}u$ for $u\in{\mathcal D}$.
\item[\rm(c)] ${\mathcal L}_p$ is the maximal restriction of ${\mathcal L}$ to $L^p(\G)$, $1<p<\infty$, that is, $\mathrm{Dom}~({\mathcal L}_p)$ is the set of all $f\in L^p(\G)$ such that the distribution derivative ${\mathcal L}f$ is in $L^p(\G)$, and ${\mathcal L}_pf={\mathcal L}f$. Also, ${\mathcal L}_p$ is the smallest closed extension of ${\mathcal L}|_{{\mathcal D}}$ on $L^p(\G)$. 
\end{enumerate}
\end{enumerate}
\end{prop}
\noindent
We call the function $h_t(x)$ as the heat kernel associated to ${\mathcal L}$ on $\G$. 
The heat kernel $h_t$ has the following estimate.
\begin{lem}[\cite{Var2}]\label{bbb} 
The heat kernel $h_t$ associated to ${\mathcal L}$ on $\G$ satisfies the following estimate:
For any $\varepsilon>0$, there exists a constant $C_{\varepsilon}>0$ such that 
\begin{align*}
h_t(x)\leq C_{\varepsilon}t^{-\frac{N}{2}}\exp(-\rho^2(x)/4(1+\varepsilon)t).
\end{align*}
for any $g\in {\mathbb G}$ and $t>0$, where $N$ is the homogeneous dimension of $\G$.
\end{lem}
\noindent
By Young's inequality and Lemma \ref{bbb}, we have the following $L^{\alpha}-L^{\beta}$ estimate.
\begin{prop}[$L^{\alpha}-L^{\beta}$ estimate]\label{17244}
Assume $1\leq \alpha\leq\beta\leq \infty$. Then there exists a positive constant $C_{p,q}$ such that
\[\|e^{-t{\mathcal L}}\varphi\|_{{\beta}}\leq Ct^{-\frac{N}{2}(\frac{1}{\alpha}-\frac{1}{\beta})}\|\varphi\|_{{\alpha}},\ t>0,\]
for any $\varphi\in L^{\alpha}(\G)$, where $N$ is the homogeneous dimension of $\G$. 
\end{prop}
\begin{proof}
Let $\varphi\in L^{\alpha}(\G)$. By Young's inequalty, we have
\begin{align*}
\|e^{-t{\mathcal L}}\varphi\|_{{\beta}} &\leq \|\varphi\|_{{\alpha}}\|h_t\|_{{\nu}},\ 1+\disp\frac{1}{\beta}=\disp\frac{1}{\alpha}+\disp\frac{1}{\nu}. 
\end{align*}
By Lemma \ref{bbb}, we have
\begin{align*}
\|h_t\|_{\nu} &\leq Ct^{-\frac{N}{2}}\disp\left(\int_{\G}e^{-\frac{c\nu\rho(g)^2}{t}}~dg\right)^{\frac{1}{\nu}}\\
                &= Ct^{-\frac{N}{2}}\disp\left(\frac{t}{c\nu}\right)^{\frac{N}{2\nu}}\disp\left(\int_{\G}e^{-{\rho(g)^2}}~dg\right)^{\frac{1}{\nu}}\\
                &= Ct^{-\frac{N}{2}(1-\frac{1}{\nu})}.
\end{align*}
Hence we have
\begin{align*}
\|e^{-t{\mathcal L}}\varphi\|_{{\beta}}\leq Ct^{-\frac{N}{2}(\frac{1}{\alpha}-\frac{1}{\beta})}\|\varphi\|_{{\alpha}}.
\end{align*} 
\end{proof}

\subsection{Hardy-Littlewood maximal function} 
\noindent

We recall the Hardy-Littlewood maximal function on ${\mathbb G}$ (in detail, see \cite{Fol}). For $f\in L^1_{\rm{loc}}(\G)$, we define a maximal function $Mf$ by
\begin{align}
M[f](x)=\disp\sup_{r>0}\disp\frac{1}{|B(x,r)|}\disp\int_{B(x,r)}|f(x^{\prime})|~dx^{\prime},\ x\in\G,
\end{align} 
where $B(x,r)=\{x^{\prime}\in{\mathbb G}\ |\ |{x^{\prime}}^{-1}\cdot x|_{\G}\leq r\}$ and $|B(x,r)|$ denotes Lebesgue measure of $B(x,r)$. As a remark, we have 
\[|B(x,r)|=|B(0,r)|=|B(0,1)|r^N\]
(see \cite{Fol}).
The maximal operator $M$ has the following properties.
\begin{prop}[\cite{Fol}, Theorem 2.4]\label{2_21_20}
Let $1<p\leq \infty$. Then the operator $M$ is bounded on $L^p(\G)$.
\end{prop}
\noindent
We also refer to \cite{Ste2} for Proposition \ref{2_21_20}.
\begin{prop}[\cite{Fef}]\label{2_21_21}
For all $\{h_k\}$ and all $1<q<\infty$, we have
\[\disp\left\|\left(\disp\sum(M[h_k])^2\right)^{1/2}\right\|_q\leq C_q \disp\left\|\left(\disp\sum|h_k|^2\right)^{1/2}\right\|_q.\]
\end{prop}

\section{Sobolev spaces on $\G$}
\noindent

Here we recall the definition and basic properties of Sobolev spaces on the stratified Lie group $\G$ first introduced in \cite{Folland2} by G. B. Folland (see also \cite{Saka}). In \cite{Fis}, V. Fischer and M. Ruzhansky gave the definition and basic properties of Sobolev spaces on graded groups, which are more general than the stratified Lie group. So we refer to \cite{Fis} and \cite{Folland2}. Especially, we adopt the definition of the Sobolev spaces in \cite{Fis}.

At first, we recall the definition of fractional powers of the sublaplacian ${\mathcal L}$.
\begin{ttt}[\cite{Fis}, \cite{Folland2}, \cite{Saka}]{\rm
Assume that $1<p<\infty$, $s>0$ and $k=[s]+1$. Then the operator ${\mathcal L}^{s}_p$ is defined by 
\begin{align}
{\mathcal L}^{s}_pf =\disp\lim_{\varepsilon\rightarrow 0}\disp\frac{1}{\Gamma(k-s)}\disp\int_{\varepsilon}^{\infty}\nu^{k-s-1}{\mathcal L}^ke^{-\nu{\mathcal L}}f~d \nu
\end{align}
on the domain of all $f\in L^p(\G)$ such that the indicated limit exists in $L^p(\G)$. The operator ${\mathcal L}^{-s}_p$ is defined by 
\begin{align}
{\mathcal L}^{-s}_pf=\disp\lim_{\eta\rightarrow \infty}\disp\frac{1}{\Gamma(s)}\disp\int_0^{\eta}\nu^{s-1}e^{-\nu{\mathcal L}}f~d\nu
\end{align}
on the domain of all $f\in L^p(\G)$ such that the indicated limit exists in $L^p(\G)$. The operator $(\mathrm{Id}+{\mathcal L}_p)^{s}$ is defined by
\begin{align}
(\mathrm{Id}+{\mathcal L}_p)^{s}f = \disp\lim_{\varepsilon\rightarrow 0}\disp\frac{1}{\Gamma(k-s)}\disp\int_{\varepsilon}^{\infty}\nu^{k-s-1}(\mathrm{Id}+{\mathcal L})^ke^{-\nu}e^{-\nu{\mathcal L}}f~d\nu
\end{align}
on the domain of all $f\in L^p(\G)$ such that the indicated limit exists in $L^p(\G)$. Also, we define the operator $(\mathrm{Id}+{\mathcal L}_p)^{-s}$ by 
\begin{align}
(\mathrm{Id}+{\mathcal L}_p)^{-s}f=\disp\frac{1}{\Gamma(s)}\disp\int_0^{\infty}\nu^{s-1}e^{-\nu}e^{-\nu{\mathcal L}}f~d\nu.
\end{align}
The operator $(\mathrm{Id}+{\mathcal L}_p)^{-s}$ is a bounded operator on $L^p$. 
}
\end{ttt}
\noindent
\begin{prop}[\cite{Fis}, \cite{Folland2}]\label{3_15_1}
Let $1<p<\infty$ and ${\mathcal M}_p$ denote either ${\mathcal L}_p$ or $\mathrm{Id}+{\mathcal L}_p$. Then, 
\begin{enumerate}
\item[\rm{(i)}] ${\mathcal M}_p^s$ is a closed operator on $L^p(\G)$ for all  $s\in{\mathbb R}$ and injective with $({\mathcal M}^s_p)^{-1}={\mathcal M}^{-s}_p$.
\item[\rm{(ii)}] If $f\in \mathrm{Dom}({\mathcal M}^{\beta}_p)\cap\mathrm{Dom}({\mathcal M}^{\alpha+\beta}_p)$, then ${\mathcal M}^{\beta}_pf\in\mathrm{Dom}({\mathcal M}^{\alpha}_p)$ and ${\mathcal M}^{\alpha}_p{\mathcal M}^{\beta}_pf={\mathcal M}^{\alpha+\beta}_pf$. ${\mathcal M}^{\alpha+\beta}_p$ becomes the smallest closed extension of ${\mathcal M}^{\alpha}_p{\mathcal M}^{\beta}_p$. 
\item[\rm{(iii)}] When $s>0$, if $f\in \mathrm{Dom}~({\mathcal M}^{s}_p)\cap L^q(\G)$, then $f\in \mathrm{Dom}~({\mathcal M}^{s}_q)$ if and only if ${\mathcal M}^{s}_pf\in L^q(\G)$, in which case ${\mathcal M}^{s}_p={\mathcal M}^{s}_q$. 
\item[\rm{(iv)}] If $s>0$, then $\mathrm{Dom}({\mathcal L}^s_p)=\mathrm{Dom}((\mathrm{Id}+{\mathcal L}_p)^s)$ .
\end{enumerate}
\end{prop}
\noindent
By Proposition \ref{3_15_1} (iii), ${\mathcal L}^s_p$ (resp. $(\mathrm{Id}+{\mathcal L}_p)^s$) agrees with ${\mathcal L}^s_q$ (resp. $(\mathrm{Id}+{\mathcal L}_q)^s$) on their common domains for $s\in{\mathbb R}$ and $1<p,q<\infty$. So we omit the subscripts on these operators except when we wish to specify domains.

Next, we recall the definition of Sobolev spaces $L^p_s(\G)$ and $\dot{L}^p_s(\G)$. We adopt the definition in \cite{Fis} (see also \cite{Folland2}).
\begin{ttt}[\cite{Fis}]\label{3_1_1}\rm 
\begin{enumerate}
\item (the inhomogeneous Sobolev space) Let $s\in{\mathbb R}$ and $1<p<\infty$. Then  we denote by $L^p_s(\G)$ the space 
of tempered distributions obtained by the completion of the Schwartz class ${\mathcal S}(\G)$ with respect to the Sobolev norm 
\[\|f\|_{L^p_s}:=\|(\mathrm{Id}+{\mathcal L})^{\frac{s}{2}}f\|_p\]
for $f\in{\mathcal S}(\G)$. 
\item (the homogeneous Sobolev space) Let $s\in{\mathbb R}$ and $1<p<\infty$. Then  we denote by $\dot{L}^p_s(\G)$ the space 
of tempered distributions obtained by the completion of ${\mathcal S}(\G)\cap\mathrm{Dom}~({\mathcal L}^{\frac{s}{2}})$ with respect to the norm 
\[\|f\|_{\dot{L}^p_s}:=\|{\mathcal L}^{\frac{s}{2}}f\|_p\]
for $f\in{\mathcal S}(\G)\cap\mathrm{Dom}~({\mathcal L}^{\frac{s}{2}}_p)$, especially, for $f\in{\mathcal S}(\G)$ if $s>0$. 

\end{enumerate}
\end{ttt}
\noindent

The Sobolev spaces $L^p_s(\G)$ and $\dot{L^p_s}(\G)$ have the following basic properties.
\begin{prop}[\cite{Fis}]\label{3_31_1}
\begin{enumerate}
\item Let $s\in {\mathbb R}$ and $1<p<\infty$. Then $L^p_s(\G)$ and $\dot{L^p_s}(\G)$ are Banach space satisfying
\[{\mathcal S}(\G)\subsetneq L^p_s(\G)\subset{\mathcal S}^{\prime}(\G)\]
and 
\[({\mathcal S}(\G)\cap\mathrm{Dom}~({\mathcal L}^{\frac{s}{2}}_p))\subsetneq \dot{L}^p_s(\G)\subsetneq{\mathcal S}^{\prime}(\G),\]
 respectively. 
\item If $s=0$ and $1<p<\infty$, then $\dot{L}^p_0(\G)=L^p_0(\G)=L^p(\G)$ with $\|\cdot\|_{\dot{L}^p_0}=\|\cdot\|_{L^p_0}=\|\cdot\|_p$.
\item if $s>0$ and $1<p<\infty$, then we have 
\[L^p_s(\G)=\dot{L}^p_s(\G)\cap L^p(\G)\]
and 
\[\|\cdot\|_{L^p_s}\sim \|\cdot\|_p+\|\cdot\|_{\dot{L}^p_s}.\]
 \end{enumerate}
\end{prop}

We summarize the useful properties of the Sobolev spaces on the stratified Lie groups $\G$ as follows.
At first, we introduce the Sobolev embedding theorem on $\G$ (for example, see \cite{Fis} Theorem 4.4.28 and see also \cite{Folland2} ).

\begin{prop}[\cite{Fis}, \cite{Folland2}]\label{embedding}
Let $\G$ be the stratified Lie groups with homogeneous dimension $N$.
\begin{enumerate}
\item If $a,b\in{\mathbb R}$ with $a<b$ and $1<p<\infty$, then the continuous inclusions $L^p_b\subset L^p_a$ holds. 
\item If $1<p<q<\infty$ and $a,b\in{\mathbb R}$ with $b-a=N\left(\frac{1}{p}-\frac{1}{q}\right)$, then we have the continuous inclusion
\begin{align}
L^p_b(\G)\subset L^q_a(\G),
\end{align}
that is, there exists a constant $C>0$ such that
\[\|f\|_{L^q_a}\leq C\|f\|_{L^p_b},\]
where $C$ depends only $a,b,p,q$, independent of $f\in L^p_b(\G)$.\\
If $1<p<\infty$ and $s>N/p$, then we have $L^p_s(\G)\subset L^{\infty}(\G)$. Furthermore, there exists a constant $C_{s,p}>0$ independent of $f$ such that 
\[\|f\|_{\infty}\leq C\|f\|_{L^p_s}.\]

\end{enumerate}
\end{prop}
\noindent

Next, we introduce the Gagliardo - Nirenberg inequality on $\G$ (see \cite{Cou}, Proposition 32 and see also \cite{Fis}, p. 247).
\begin{prop}[\cite{Cou},\cite{Fis}]\label{3_11_5}
Let $\alpha,\gamma>0$, $1<p,r<\infty$, $1<q\leq \infty$ and $0<\theta<1$ be such that $\gamma=\theta\alpha$ and $r^{-1}=\theta p^{-1}$. Then there exists $C>0$ such that for any $f\in \dot{L}^p_{\alpha}(\G)\cap L^q(\G)$,  
\begin{align}
\|f\|_{\dot{L}^r_\gamma}\leq C\|f\|^{\theta}_{\dot{L}^p_{\alpha}}\|f\|^{1-\theta}_{q}.
\end{align}
\end{prop}
\noindent

Next, we introduce the generalized Leibniz rule on $\G$ (see \cite{Cou}, Theorem 4.)
\begin{prop}[\cite{Cou}]\label{3_11_4}
Let $s\geq 0$ and $1<p_1,q_2\leq \infty$, $1<r,p_2,q_1<\infty$, $r^{-1}=p_i^{-1}+q^{-1}_i$ for $i=1,2$. 
Then there exists $C>0$ such that for any $f\in L^{p_1}(\G)\cap\dot{L}^{p_2}_s(\G)$ and $g\in L^{q_2}(\G)\cap\dot{L}^{q_1}_s(\G)$, $fg\in \dot{L}^r_s(\G)$ and 
\begin{align}
\|fg\|_{\dot{L}^r_s}\leq C(\|f\|_{p_1}\|g\|_{\dot{L}^{q_1}_s}+\|g\|_{q_2}\|f\|_{\dot{L}^{p_2}_s}).
\end{align}
\end{prop} 
\noindent

Finally, we give the $L_{s+\theta}^p - L_s^p$ estimate on $\G$.
\begin{prop}[$L_{s+\theta}^p - L_s^p$ estimate]\label{L4}
Assume $1< p< \infty$, $s\geq 0$ and $\theta\geq 0$. Then there exists a positive constant $C$ (independent of t) such that
\[\|e^{-t{\mathcal L}}\varphi\|_{L^p_{s+\theta}}\leq Ct^{-\frac{\theta}{2}}\|\varphi\|_{L^p_s},\ 0<t<1,\]
for any $\varphi\in L^{p}_s(\G)$, where $e^{-t{\mathcal L}}f(x)=(f*h_t)(x)$ and $h_t$ is the heat kernel associated to a sublaplacian ${\mathcal L}$ on $\G$. 
\end{prop}
\noindent
\begin{proof}
By Corollary 3.6 in \cite{Folland2}, for any positive integer $k$, ${\mathcal L}^kh_t\in L^1(\G)$ and ${\mathcal L}^kh_t(g)=t^{-k-N/2}{\mathcal L}^kh_1(t^{-1/2}g)$. So we have
\begin{align}
\|{\mathcal L}^kh_t\|_1=t^{-k}\|{\mathcal L}^kh_1\|_1\label{5-1}
\end{align}
for $t>0$.
On the other hand, since we can commute $e^{-t{\mathcal L}}$ with ${\mathcal L}^k$ on $\mathrm{Dom}~({\mathcal L}^k)$, we have
\begin{align}
{\mathcal L}^ke^{-t{\mathcal L}}f=e^{-t{\mathcal L}}({\mathcal L}^kf). \label{5-2}
\end{align}
By \eqref{5-1} and \eqref{5-2}, the following lemma concerning on homogeneous Sobolev spaces $\dot{L}^p_s(\G)$ holds.
\begin{lem}\label{5-3}
Let $1<p<\infty$ and $s\geq 0$. Then for any $\theta\geq 0$, there exists a positive constant $C_{\theta}$ such that 
\begin{align}
\|e^{-t{\mathcal L}}f\|_{\dot{L}^p_{s+\theta}}\leq Ct^{-\frac{\theta}{2}}\|f\|_{\dot{L}^p_s}, t>0
\end{align}
for $f\in\dot{L}^p_s(\G)$.
\end{lem}
If we assume that Lemma \ref{5-3} holds, then since 
\[\|f\|_{L^p_s}\sim\|f\|_p+\|f\|_{\dot{L}^p_s},\]
by Young's inequality, we have
\begin{align}
\|e^{-t{\mathcal L}}f\|_{L^p_{s+\theta}} &\sim \|e^{-t{\mathcal L}}f\|_p+\|e^{-t{\mathcal L}}f\|_{\dot{L}^p_{s+\theta}}\\
                                                   &\leq \|f\|_p\|h_t\|_1+Ct^{-\frac{\theta}{2}}\|f\|_{\dot{L}^p_s}\\
                                                   &= \|f\|_p+Ct^{-\frac{\theta}{2}}\|f\|_{\dot{L}^p_s}\\
                                                   &\leq \tilde{C} t^{-\frac{\theta}{2}}\|f\|_{{{L}^p_s}}
\end{align}
for $t\in (0,1)$, where $\tilde{C}>0$ is a constant, independent of $t$. It remains to show the proof of Lemma \ref{5-3}. If $\theta=0$, it is clear. So let $\theta> 0$, $s\geq 0$ and $t>0$. 
By the same reason for \eqref{5-2}, we have ${\mathcal L}^{\frac{s}{2}}e^{-t{\mathcal L}}f=e^{-t{\mathcal L}}{\mathcal L}^{\frac{s}{2}}f$. Hence we obtain 
\begin{align}
{\mathcal L}^{\frac{s+\theta}{2}}e^{-t{\mathcal L}}f=({\mathcal L}^{\frac{s}{2}}f)*({\mathcal L}^{\frac{\theta}{2}}h_t).
\end{align}
On the other hand, by \eqref{5-1}, ${\mathcal L}^{\frac{\theta}{2}}h_t$ is well-defined for any $\theta> 0$ in $L^1(\G)$ and we have
\begin{align}
\|{\mathcal L}^{\frac{\theta}{2}}h_t\|_1 &= \disp\left\|\disp\frac{1}{\Gamma(k-\frac{\theta}{2})}\disp\int_0^{\infty}\nu^{k-\frac{\theta}{2}-1}{\mathcal L}^kh_{t+\nu}d\nu\right\|_1\\
                                                        &\leq \disp\frac{1}{\Gamma(k-\frac{\theta}{2})}\disp\int_0^{\infty}\nu^{k-\frac{\theta}{2}-1}(t+\nu)^{-k}\|{\mathcal L}^kh_1\|_1d\nu\\
                                                        &\leq \disp\frac{Ct^{-k}}{\Gamma(k-\frac{\theta}{2})}\disp\int_0^{\infty}\nu^{k-\frac{\theta}{2}-1}\left(1+\disp\frac{\nu}{t}\right)^{-k}d\nu\\
                                                        &= \disp\frac{Ct^{-\frac{\theta}{2}}}{\Gamma(k-\frac{\theta}{2})}\disp\int_0^{\infty}\nu^{k-\frac{\theta}{2}-1}\left(1+\nu\right)^{-k}d\nu\\
                                                        &\leq C_{\theta}t^{-\frac{\theta}{2}}.
\end{align}
By Young's inequality, we obtain
\begin{align}
\|{\mathcal L}^{\frac{s+\theta}{2}}e^{-t{\mathcal L}}f\|_p &= \|({\mathcal L}^{\frac{s}{2}}f)*({\mathcal L}^{\frac{\theta}{2}}h_t)\|_p\\
                                                                  &\leq \|{\mathcal L}^{\frac{s}{2}}f\|_p\|{\mathcal L}^{\frac{\theta}{2}}h_t\|_1\\
                                                                  &\leq C_{\theta}t^{-\frac{\theta}{2}}\|{\mathcal L}^{\frac{s}{2}}f\|_p.
\end{align}
Therefore we can see that for any $\theta\geq 0$, there exists a positive constant $C$ such that
\begin{align}
\|e^{-t{\mathcal L}}f\|_{\dot{L}^p_{s+\theta}}\leq Ct^{-\frac{\theta}{2}}\|f\|_{\dot{L}^p_s},\ t>0,
\end{align} 
for $s\geq 0$. \\
\end{proof}

\section{The proof of Theorem \ref{th1}}\label{17866}
\noindent

We assume that Theorem \ref{th2} holds and $u_0$ belongs to $L^p_s(\G)$, where $s$ satisfies \eqref{1-5} and $s>s_c=N/p-2/(\alpha-1)$. Set
\[\Phi(u)(t,x)=\disp\int_0^te^{-(t-\tau){\mathcal L}}F(u(\tau,x))~d\tau\]
and define the exponent $\tilde{p}$ by
\begin{align}
\disp\frac{1}{\tilde{p}}=\disp\frac{1}{p}-\disp\frac{s}{N}.\label{asa}
\end{align}
Then by \eqref{1-5} and $s>s_c$, we have 
\begin{align}
\tilde{p}>\alpha\ \ \text{and}\ \ \tilde{p}>p_c=N(\alpha-1)/2.
\end{align}
Furthermore we define the space
\[X=C([0,T_0);L^{\tilde{p}}(\G))\ \ \text{and}\ \ Y=C([0,T_0);L^p_s(\G))\]
with the norms
\[\|u\|_X:=\disp\sup_{0\leq t<T_0}\|u\|_{\tilde{p}},\ \|u\|_Y:=\disp\sup_{0\leq t<T_0}\|u\|_{L^p_s}\]
for some $T_0>0$. Then by Proposition \ref{embedding} (Sobolev embedding theorem), we have 
the including relationship
\[Y\hookrightarrow X.\]
Let us consider the sequence of functions 
\[u^0=e^{-t{\mathcal L}}u_0(x),\ u^{j+1}=u^0+\Phi(u^j)\]
for $j=0,1,2,\cdots$. First by a standard fixed point argument, we will prove that $\{u^j\}$ converges strongly in $X$ to a function $u$ which verifies \eqref{1-4}. Second, by Theorem \ref{th2}, we will show that $u$ also belongs to $Y$.

By Proposition \ref{17244} ($L^\alpha-L^\beta$ estimate )  and  Proposition \ref{embedding} (Sobolev embedding theorem),
there exists a constant $C>0$ (independent of $u_0$ and $t$) such that
\begin{align}
\|u^0\|_X\leq \|u_0\|_{\tilde{p}}\leq C\|u_0\|_{L^p_s}.\label{6-1}
\end{align} 
Let $u,v\in X$. Then by Proposition \ref{17244} ($L^\alpha-L^\beta$ estimate)  and H$\mathrm{\ddot{o}}$lder's inequality, we obtain the estimate
\begin{align}
\|\Phi(u)(t)-\Phi(v)(t)\|_{\tilde{p}} &\leq \disp\int_0^t\|e^{-(t-\tau){\mathcal L}}(F(u)(\tau)-F(v)(\tau))\|_{\tilde{p}}~d\tau\\
                                                  &\leq M\disp\int_0^t(t-\tau)^{-\beta}\|F(u)(\tau)-F(v)(\tau)\|_{\frac{\tilde{p}}{\alpha}}~d\tau\\
                                                  &\leq M\disp\int_0^t(t-\tau)^{-\beta}\|u(\tau)-v(\tau)\|_{\tilde{p}}\left(\|u(\tau)\|_{\tilde{p}}^{\alpha-1}+\|v(\tau)\|^{\alpha-1}_{\tilde{p}}\right)~d\tau\\
                                                  &\leq MBt^{1-\beta}\|u-v\|_X\left(\|u\|_X^{\alpha-1}+\|v\|_X^{\alpha-1}\right),\label{6-2}
\end{align}
where $M$ is some positive constant, $\beta=\disp{N(\alpha-1)}/{(2\tilde{p})}<1$ and $B=\disp{1}/{(1-\beta)}$. From $F(0)=0$, \eqref{6-1} and \eqref{6-2}, we deduce that
\begin{align}
\|u^{j+1}\|_X\leq \|u^0\|_X+MBT_0^{1-\beta}\|u^j\|_X^{\alpha}\label{6-3}
\end{align}
and
\begin{align}
\|u^{j+1}-u^j\|_X\leq MBT_0^{1-\beta}\|u^j-u^{j-1}\|_X\left(\|u^{j}\|^{\alpha-1}_X+\|u^{j-1}\|^{\alpha-1}_X\right)\label{6-4}
\end{align}
By \eqref{6-1} and \eqref{6-3}, for any $j\geq 0$, we have 
\begin{align}
\|u^j\|_X< 2C\|u_0\|_{L^p_s}\label{6-5}
\end{align}
for 
\begin{align}
T_0<\left(2^{\alpha}C^{\alpha-1}MB\right)^{-\frac{1}{1-\beta}}\|u_0\|_{L^p_s}^{-\frac{\alpha-1}{1-\beta}}\label{6-6}
\end{align}
and by \eqref{6-4} and \eqref{6-5}, we obtain
\begin{align}
\|u^{j+1}-u^{j}\|_X\leq 2(2C)^{\alpha-1}MBT_0^{1-\beta}\|u_0\|_{L^p_s}^{\alpha-1}\|u^j-u^{j-1}\|_X\label{6-7}
\end{align}
for \eqref{6-6}, that is, $2(2C)^{\alpha-1}MBT_0^{1-\beta}\|u_0\|_{L^p_s}^{\alpha-1}<1$. By Banach's fixed point theorem, there is a function $u$ such that 
\[\disp\lim_{j\rightarrow\infty}u_j=u\ \text{in}\ X,\]
which solves $u=u_0+\Phi(u)$.

We proceed to prove that this solution $u$ also belongs to $Y$. Then by  Proposition \ref{L4} ($L^p_{s+\theta}-L^p_s$ estimate)  and Theorem \ref{th2}, we have
\begin{align}
\|\Phi(u^j)(t)\|_{L^p_s} &\leq \disp\int_0^t\|e^{-(t-\tau){\mathcal L}}F(u^j)(\tau)\|_{L^p_s}~d\tau\\
                                &\leq M^{\prime}\disp\int_0^t(t-\tau)^{-\frac{s-s_{\alpha}}{2}}\|F(u^j)(\tau)\|_{L^p_{s_{\alpha}}}~d\tau\\
                                &\leq M^{\prime}\disp\int_0^t(t-\tau)^{-\beta}\|u^j(\tau)\|_{L^p_s}^{\alpha}~d\tau\\
                                &\leq M^{\prime}BT^{1-\beta}_0\|u^j\|_Y^{\alpha}\label{6-8}
\end{align} 
for some constants $M^{\prime}>0$ since $\disp{(s-s_{\alpha})}/{2}=\beta<1$. By \eqref{6-8}, we have
\begin{align}
\|u^{j+1}\|_Y\leq \|u^0\|_{Y}+M^{\prime}BT^{1-\beta}_0\|u^j\|^{\alpha}_Y.\label{6-9}
\end{align}
By \eqref{6-9}, if $2M^{\prime}BT_0^{1-\beta}(2C)^{\alpha-1}\|u_0\|^{\alpha-1}_{L^p_s}<1$, we obtain 
\[\|u^j\|_Y<2C\|u_0\|_{L^p_s}\]
for any $j\geq 0$. Hence if $T_0$ satisfies the condition
\[2M^{\prime}BT_0^{1-\beta}(2C)^{\alpha-1}\|u_0\|^{\alpha-1}_{L^p_s}<1,\]
we can see that $\|u^j\|_Y$ remain bounded. So we can always take a subsequence $\{u^{j_k}\}$ such that 
this converges weakly - $\ast$ to a function $\tilde{u}\in Y$. Now the $u^{j_k}$ converge to $u$ and converge to $\tilde{u}\in {\mathcal D}^{\prime}((0,T_0)\times\G)$. Thus, $u$ coincide with $\tilde{u}$. Therefore we have proved the existence of a  solution of \eqref{1-4} in $C([0,T_0);L^p_s(\G))$.

Finally, we prove the uniqueness. Let $u,v\in Y$ be two solutions for the same initial value $u_0$ and let $0\leq T<T_0$. Then since $u-v=\Phi(u)-\Phi(v)$, we have the estimate
\begin{align}
\|u-v\|_X\leq 2MBT^{1-\beta}K^{\alpha-1}\|u-v\|_X,
\end{align}
where $K=\disp\sup_{t\in[0,T_0)}\{\|u(t)\|_{\tilde{p}}, \|v(t)\|_{\tilde{p}}\}$. Now we can take $T$ small such that $2MBT^{1-\beta}K^{\alpha-1}<1$. This implies that $u=v$ on $[0,T]$. Finitly repeating the same argument, we can see that $u=v$ on $[0,T_0)$.\\

\section{The proof of Theorem \ref{th2}}
\subsection{The case $s_{\alpha}\leq 0$}
\noindent

We suppose that $\disp{N}/{p}-\disp{N}/{\alpha}<s<N/p$, $s\geq 0$ and N1) are fulfilled, that is, we suppose that $s_{\alpha}\leq 0$ and $|F(u)|\leq C|u|^{\alpha}$ for $\alpha>1$. Let $1<p<\infty$ and $u\in L^p_s(\G)$. Then by Proposition \ref{embedding} (Sobolev embedding theorem), we have 
\begin{align}
L^p_s(\G) \subset L^{(\frac{1}{p}-\frac{s}{N})^{-1}}(\G). \label{th2-1}
\end{align}
Since $\disp{N}/{p}-\disp{N}/{\alpha}<s<N/p$, we obtain $\disp\left(\disp{1}/{p}-\disp{s}/{N}\right)^{-1}>\alpha$. Hence by \eqref{th2-1}, we have 
\begin{align}
\|F(u)\|_{(\frac{\alpha}{p}-\frac{s\alpha}{N})^{-1}}\leq C\|u\|^{\alpha}_{(\frac{1}{p}-\frac{s}{N})^{-1}}\leq C\|u\|^{\alpha}_{L^p_s}.\label{th2-2}
\end{align}
On the other hand, we obtain
\[\disp\frac{1}{p}-\disp\frac{s_{\alpha}}{N}=\disp\frac{\alpha}{p}-\disp\frac{s\alpha}{N}.\]
Since $s_{\alpha}\leq 0$ and $\disp\left({\alpha}/{p}-\disp{s\alpha}/{N}\right)^{-1}>1$, by Proposition \ref{embedding} (Sobolev embedding theorem), we have
\begin{align}
L^{(\frac{\alpha}{p}-\frac{s\alpha}{N})^{-1}}(\G)\subset L^p_{s_{\alpha}}(\G).\label{th2-3}
\end{align}
Therefore by \eqref{th2-2} and \eqref{th2-3}, we obtain 
\[\|F(u)\|_{L^p_{s_{\alpha}}}\leq C\|F(u)\|_{L^{(\frac{\alpha}{p}-\frac{s\alpha}{N})^{-1}}}\leq C\|u\|^{\alpha}_{L^p_s}.\]
Hence we have
\[\|F(u)\|_{L^p_{s_{\alpha}}}\leq C\|u\|^{\alpha}_{L^p_s}.\]
\hspace{13.3cm}$\square$

\subsection{The case $s_{\alpha}>0$}
\noindent

To prove Theorem \ref{th2} in the case $s_{\alpha}>0$, we need the following generalized chain rule. 
\begin{lem}\label{3_11_2}
Assume that $F\in C^l$, $l\in{\mathbb N}$, with
\[F(0)=0,\ |F^{(j)}(x)|\leq C|x|^{\alpha-j},\ \alpha>1,\ \alpha\geq l\]
for $j=1,2,\cdots, l$, $0\leq\delta\leq l$, $1<p,q<\infty$, $1<r\leq \infty$ and $p^{-1}=(\alpha-1)r^{-1}+q^{-1}$. If $u\in L^r(\G)\cap\dot{L}^q_{\delta}(\G)$,  then ${\mathcal L}^{\frac{\delta}{2}}F(u)\in L^p(\G)$ and 
\begin{align}
\|{\mathcal L}^{\frac{\delta}{2}}F(u)\|_p \leq C\|u\|_r^{\alpha-1}\|{\mathcal L}^{\frac{\delta}{2}}u\|_q.\label{3_11_3}
\end{align}
\end{lem}
\noindent
We will prove this lemma below, following Kato's work on ${\mathbb R}^n$ in \cite{Kato}.

We suppose that $\mathrm{max}\disp\left\{0,\disp{N}/{p}-\disp{N}/{\alpha}\right\}<s<N/p$ and N2) are fulfilled, that is, we suppose that $0<s_{\alpha}<\alpha-1$ and $|F^{(j)}(u)|\leq C|u|^{\alpha-j}$ for $\alpha>1$ and $j=1,\cdots,[\alpha$]. Let $u(x)\in L^p_s(\G)$, $1<p<\infty$.
By Proposition \ref{3_31_1} (iii), $\|F(u)\|_{L^p_{s_{\alpha}}}\sim\|F(u)\|_p+\|{\mathcal L}^{\frac{s_{\alpha}}{2}}F(u)\|_p$. So we will show the following estimates,
\[ \|F(u)\|_{p}\leq C\|u\|^{\alpha}_{L^p_s}\ \text{and}\ \|{\mathcal L}^{\frac{s_{\alpha}}{2}}F(u)\|_{p}\leq C\|u\|^{\alpha}_{L^p_s}.\]
At first, we will show the estimate $\|{\mathcal L}^{\frac{s_{\alpha}}{2}}F(u)\|_{p}\leq C\|u\|^{\alpha}_{L^p_s}$.
Since $0<s_{\alpha}<\alpha-1\leq [\alpha]$,   
by Lemma \ref{3_11_2} (the generalized chain rule) and Proposition \ref{embedding} (Sobolev embedding theorem), we have
\begin{align}
\|{\mathcal L}^{\frac{s_{\alpha}}{2}}F(u)\|_p 
\leq C\|u\|^{\alpha-1}_r\|{\mathcal L}^{\frac{s_{\alpha}}{2}}u\|_q
\leq C\|u\|^{\alpha-1}_{L^p_{\tau_1}}\|u\|_{L^p_{\tau_2}}
\label{2_22_31}
\end{align}
where 
\begin{align}
\frac{1}{p}=\frac{\alpha-1}{r}+\frac{1}{q},\label{2_26_5}
\end{align}
\begin{align}
\  \tau_1=N\left(\frac{1}{p}-\frac{1}{r}\right),\ \tau_2=s_{\alpha}+N\left(\frac{1}{p}-\frac{1}{q}\right). \label{2_26_1}
\end{align}
We take $r$ and $q$ such that $\tau_1=\tau_2$, that is, 
\begin{align}
N\left(\frac{1}{p}-\frac{1}{r}\right)=s_{\alpha}+N\left(\frac{1}{p}-\frac{1}{q}\right)(=\tilde{s}).\label{3_12_7}
\end{align}
By \eqref{2_26_5} and \eqref{3_12_7}, we have 
\begin{align}
\disp\frac{1}{q}=\disp\frac{1}{\alpha}\left\{\frac{1}{p}+(\alpha-1)\disp\frac{s_{\alpha}}{N}\right\},\ \disp\frac{1}{r} = \disp\frac{1}{\alpha}\left(\frac{1}{p}-\frac{s_{\alpha}}{N}\right)
\end{align}
and
\begin{align}
\tilde{s}=\disp\frac{1}{\alpha}\left\{s_{\alpha}+(\alpha-1)\frac{N}{p}\right\}.
\end{align}
Therefore by \eqref{2_22_31}, we obtain
\[\|{\mathcal L}^{\frac{s_{\alpha}}{2}}F(u)\|_p\leq C\|u\|_{L^p_{\tilde{s}}}^{\alpha}.\]
Hence we obtain \eqref{2121} for $s$ satisfying the following condition
\begin{align}
s\geq \disp\frac{1}{\alpha}\left\{s_{\alpha}+(\alpha-1)\frac{N}{p}\right\}.\label{2_26_11}
\end{align}
By \eqref{s_alpha}, we can see that \eqref{2_26_11} holds for the equal. 

Furthermore by a straightforward calculation, we can also see that 
$r$, $q$ are in $(1,\infty)$, respectively, and satisfy the condition for Sobolev embedding theorem,
$p< r,\ p< q$. 

Secondly, we will show the estimate $\|F(u)\|_p\leq C\|u\|_{L^p_s}^{\alpha}$. Indeed, by Proposition \ref{embedding} (Sobolev embedding theorem), we have
\begin{align}
\|F(u)\|_p &\leq C\|u\|^{\alpha}_{\alpha p}\leq C\|u\|^{\alpha}_{L^p_{\Xi}},
\end{align}
where $\Xi=N(1/p-1/(\alpha p))$. Since $s$ satisfies \eqref{2_26_11} with an equal and $s_{\alpha}>0$, we have $s> \Xi$. Hence we obtain \eqref{2121}. This completes the proof of Theorem \ref{th2}.\\
\hspace{13.3cm}$\square$

\section{The proof of \hbox{Lemma \ref{3_11_2}}}
\noindent

By the spectral theorem, the sublaplacian ${\mathcal L}$ on $\G$ satisfies a spectral resolution
\[{\mathcal L}=\disp\int_0^{\infty}\lambda dE_{\lambda},\]
where $dE_{\lambda}$ is the projection measure. If $\Theta$ is a bounded Borel measure function on ${\mathbb R}_+$, then the operator
\[\Theta({\mathcal L})=\disp\int_0^{\infty}\Theta({\lambda})dE_{\lambda}\]
is bounded on $L^2(\G)$. Furthermore by the Schwartz kernel theorem, there exists a tempered distribution kernel $K_{\Theta({{\mathcal L}})}$ on $\G$ such that 
\[\Theta({\mathcal L})f=f*K_{\Theta({{\mathcal L}})}\]
for any $f\in{\mathcal S}(\G)$. It is known that if $\Theta\in{\mathcal S}({\mathbb R}_+)$, then the distribution kernel $K_{\Theta({{\mathcal L}})}$ of the operator $\Theta({\mathcal L})$ belongs to ${\mathcal S}(\G)$
(see, \cite{Fur}, \cite{Hu} and \cite{Hul2}).
Furthermore, for any $\Theta\in{\mathcal S}({\mathbb R}_+)$ and $t>0$, it also holds that the kernel $K_{\Theta(t{{\mathcal L}})}$
of the operator $\Theta(t{\mathcal L})$ belongs to ${\mathcal S}(\G)$ (see \cite{Hu}). So it is well-defined that for any $f\in{\mathcal S}^{\prime}(\G)$,
\[\Theta(t{\mathcal L})f(x)=(f*K_{\Theta({{t\mathcal L}})})(x),\ x\in\G.\]

Following \cite{Hu}, we also recall the homogeneous Tribel-Lizorkin spaces $\dot{F}^s_{pq}$ on $\G$.
\begin{ttt}\label{7-1-1}\rm
Let $\varphi\in C^{\infty}({\mathbb R}_+)$ such that $\mathrm{supp}~\varphi\subset[1/4, 4]$, $|\varphi(\lambda)|\geq c>0$ for $\lambda\in[(3/5)^2, (5/3)^2]$. For $s\in{\mathbb R}$, $0<p<\infty$ and $0<q\leq \infty$, the homogeneous Triebel-Lizorkin space $\dot{F}^s_{pq}(\G)$ is defined as the set of all $f\in {\mathcal S}^{\prime}(\G)/{\mathcal P}$ for which 
\[\|f\|_{\dot{F}^s_{pq}}=\disp\left\|\left(\disp\sum_{j\in{\mathbb Z}}2^{jsq}|\varphi(2^{-2j}{{\mathcal L}})f|^q\right)^{\frac{1}{q}}\right\|_p<\infty,\]
where ${\mathcal P}$ denotes the space of all polynomials on $\G$.
\end{ttt}
\noindent
By the general theory developed in \cite{Hu}, it is known that the definition of the space $\dot{F}^s_{pq}(\G)$ is independent of the choice of $\varphi$, as long as $\varphi$ satisfy all the condition in Definition \ref{7-1-1}.

\begin{lem}\label{7-2-1}
Let $s\in{\mathbb R}$ and $1< p<\infty$. Then the following identification is valid:
\[\dot{F}^s_{p2}(\G)\cong\dot{L}^p_s(\G)\]
with equivalent norm $\|f\|_{\dot{F}^s_{p2}}\sim\|f\|_{\dot{L}^p_s}$.
\end{lem}
\begin{proof}
By the lifting property of $\dot{F}^s_{pq}(\G)$ and the fact $\dot{F}^0_{p2}(\G)=L^p(\G)$ for $1<p<\infty$ ($\cite{Hu}$, Theorem 13 and Corollary 15), we have
\begin{align}
\|f\|_{\dot{F}^s_{p2}}\sim \|{\mathcal L}^{\frac{s}{2}}f\|_{p}=\|f\|_{\dot{L}^p_s}
\end{align}
for $1<p<\infty$.
\end{proof}

We give the fractional chain rule on $\G$ following in \cite{Chr}.
Let $\psi\in C^{\infty}({\mathbb R}_+)$ satisfying the same condition for $\varphi$ in Definition \ref{7-1-1} and $\sum_{j\in{\mathbb Z}}\psi_j(\lambda)=1$, where $\psi_j(\lambda)=\psi(2^{-2j}\lambda)$. Then 
there exists a function $\Psi\in S(\G)$ satisfying the following conditions:
\begin{enumerate}
\item[\rm{(i)}]  $\psi({\mathcal L})f(x)=(f*\Psi)(x)$ for any $f\in \dot{F}^s_{pq}$,
\item[\rm{(ii)}] for any $l\in{\mathbb N}$, there exists a constant $C_l>0$ such that
\begin{align}
|\Psi(x)|+|X\Psi(x)|\leq C_l(1+|x|_{\G})^{-l}, \label{2_22_111}
\end{align}
\item[\rm{(iii)}] \begin{align}
 \disp\int_{\G}\Psi(x)~dx=0.
 \end{align}
\end{enumerate}
\noindent
Put $\Psi_j(x)=2^{Nj}\Psi(\delta_{2^j}(x))$. Then $\Psi_j$ satisfies 
\begin{align}
|\Psi_j(x)|+2^{-j}|X\Psi_j(x)|\leq C_l2^{Nj}(1+2^j|x|_{\G})^{-l}\label{2_22_1}
\end{align}
for any $l\in{\mathbb N}$ and 
 \begin{align}
 \disp\int_{\Hei}\Psi_j(x)~dx=0.\label{2_22_2}
 \end{align}
Moreover we can take $\tilde\psi\in C^{\infty}({\mathbb R}_+)$ such that $\mathrm{supp}~\tilde\psi\subset[2^{-4},2^4]$ and $\tilde\psi\psi=\psi$. Then we have $I=\sum_{j=-\infty}^{\infty}\psi_j=\sum_{j=-\infty}^{\infty}\tilde\psi_j\psi_j$, where $\psi_j(\lambda)=2^{Nj}\psi(\delta_{2^j}(\lambda))$. The kernel $\tilde\Psi_j(x)$ of the operator $\tilde\psi_j({\mathcal L})$ satisfies \eqref{2_22_1} and \eqref{2_22_2}. Let $\eta_{j,l}(x)=2^{Nj}(1+2^{j}|x|_{\G})^{-l}$ for $l\in{\mathbb N}$. Then the decreasing function $\eta_{j,l}$ is radial. 

We need the following lemma
to prove the fractional chain rule on $\G$.
\begin{lem}\label{2_21_2}
For any $h$, there exist positive constants $C$ and $D$ such that 
\begin{align}
|\tilde{\psi}_j({\mathcal L})h(x^{\prime})-\tilde{\psi}_j({\mathcal L})h(x)|\leq C
\begin{cases}
2^j|{x^{\prime}}^{-1}\cdot x|_\G M[h](x),\ \text{if}\ |{x^{\prime}}^{-1}\cdot x|_{\G}\leq D2^{-j},\\
M[h](x)+M[h](x^{\prime})\ \text{for\ any}\ x,x^{\prime}\in\G,
\end{cases}
\end{align}
where the operator $M$ is the maximal operator.
\end{lem}
\begin{proof}
For any $x,x^{\prime}\in\G$, we have 
\begin{align}
|\tilde{\psi}_j({\mathcal L})h(x)| &\leq \disp\int_{\G}|h(\omega)||\tilde{\Psi}_j(\omega^{-1}\cdot x)|~d\omega\\
                                                         &\leq C(|h|*\eta_{j,N+2})(x).
\end{align}
If the inequality 
\begin{align}
|h|*\eta_{j,N+2}(x)&\leq CM[h](x).  \label{2_21_3}
\end{align}
holds, then we obatin for any $x,x^{\prime}\in\G$, 
\[|\tilde{\psi}_j({\mathcal L})h(x^{\prime})-\tilde{\psi}_j({\mathcal L})h(x)|\leq C(M[h](x)+M[h](x^{\prime})).\]
Now we prove \eqref{2_21_3} (following in \cite{Ste2}, pp. 62--63). Let $f=|h|$. Then by the polar coordinate transformation, we have
\begin{align}
\disp\int_{|x|_{\G}\leq r}f(x)~dx &= \disp\int_0^r\left\{\disp\int_{{\mathbb S}}f(\rho y)d\alpha(y)\right\}\rho^{N-1}~d\rho\\
                                                     &=\disp\int_0^r\lambda(\rho)\rho^{N-1}~d\rho \label{2_21_4}
                                                     =: \Lambda(r),
\end{align}
where $\lambda(\rho)=\disp\int_{{\mathbb S}}f(\rho y)d\alpha(y)$. On the other hand, we obtain
\begin{align}
\disp\int_{|x|_{\G}\leq r}f(x)~dx &\leq  |B(0,1)|r^NM[f](0). \label{2_21_5}
\end{align}
By \eqref{2_21_4} and \eqref{2_21_5}, we have
\begin{align}
\Lambda(r)\leq |B(0,1)|r^NM[f](0). \label{2_21_6} 
\end{align}
Hence by \eqref{2_21_6}, we can see that
\begin{align}
\Lambda(A)\eta_{j,N+2}(A)\rightarrow 0 \label{2_21_7}
\end{align} 
as $A\rightarrow \infty$ and
\begin{align}
\Lambda(\varepsilon)\eta_{j,N+2}(\varepsilon)\rightarrow 0 \label{2_21_8}
\end{align} 
as $\varepsilon\rightarrow 0$.
Since $\eta_{j,N+2}({x^{\prime}}^{-1})=\eta_{j,N+2}(x^{\prime})$ and $\eta_{j,N+2}$ is radial, by \eqref{2_21_6}, \eqref{2_21_7} and \eqref{2_21_8}, we have
\begin{align}
(f*\eta_{j,N+2})(0) &= \disp\int_{\G}f(x^{\prime})\eta_{j,N+2}(x^{\prime})~dx^{\prime}\\
                      &= \disp\int_0^{\infty}\lambda(\rho)\eta_{j,N+2}(\rho)\rho^{N-1}~d\rho\\
                      &= \disp\lim_{A\rightarrow\infty\atop\varepsilon\rightarrow 0}\disp\int_{\varepsilon}^A\lambda(\rho)\eta_{j,N+2}(\rho)\rho^{N-1}~d\rho\\
                      &= \disp\lim_{A\rightarrow\infty\atop\varepsilon\rightarrow 0}\disp\left\{\left(\Lambda(A)\eta_{j,N+2}(A)-\Lambda(\varepsilon)\eta_{j,N+2}(\varepsilon)\right)-\disp\int_{\varepsilon}^{A}\Lambda(\rho)\eta_{j,N+2}^{\prime}(\rho)~d\rho\right\}\\
                      &= \disp\int_{0}^{\infty}\Lambda(\rho)\left(-\eta_{j,N+2}^{\prime}(\rho)\right)~d\rho\\
                      &\leq |B(0,1)|M[f](0)\disp\int_0^{\infty}\rho^N|\eta^{\prime}_{j,N+2}(\rho)|~d\rho\\
                      &\leq C_NM[f](0).
\end{align}
Put $L_af(g):=f(ag)$ for $a,g\in\G$. Then we have
\[(f*\eta_{j,N+2})(x)=(L_xf)*\eta_{j,N+2}(0)\leq C_NM[L_xf](0)=C_NL_xM[f](0)=C_NM[f](x).\]
Hence we have proved \eqref{2_21_3}. 

We proceed to the case $|{x^{\prime}}^{-1}\cdot x|_{\G}\leq D2^{-j}$ for some $D>0$. Then we have
\begin{align}
&|\tilde\psi_j({\mathcal L})h(x^{\prime})-\tilde\psi_j({\mathcal L})h(x)|\\
 &\leq \disp\int_{\G}|h(\omega)||\tilde{\Psi}_j(\omega^{-1}\cdot x^{\prime})-\tilde{\Psi}_j(\omega^{-1}\cdot x)|~d\omega\\
        &\leq C_N2^{(N+1)j}|{x^{\prime}}^{-1}\cdot x|_{\G}\disp\int_{\G}|h(\omega)|(1+2^j|\omega^{-1}\cdot x|_{\G})^{-(N+2)}~d\omega\\
        &\leq C_N 2^{j}|{x^{\prime}}^{-1}\cdot x|_{\G}M[h](x).
\end{align}
In the second- to -last inequality, we have invoked the following estimate: There exists a constant $D>0$ such that if  $|{x^{\prime}}^{-1}\cdot x|_{\G}\leq D2^{-j}$, then for any $l\in{\mathbb N}$, there exists a constant $C_l>0$ such that 
\begin{align}
&|\tilde{\Psi}_j(\omega^{-1}\cdot x^{\prime})-\tilde{\Psi}_j(\omega^{-1}\cdot x)|\\
&\leq C_l2^{(N+1)j}|{x^{\prime}}^{-1}\cdot x|_{\G}(1+2^j|\omega^{-1}\cdot x|_{\G})^{-l}.\label{2_21_9}
\end{align}
We will show the estimate \eqref{2_21_9}.  First, assume the case $j=0$. Put $y=\omega^{-1}\cdot x$ and $z={x^{\prime}}^{-1}\cdot x$. Then by mean value theorem (see \cite{Fol}), there exist positive constants $C$ and $b$ such that for any $y,z\in\G$, 
\begin{align}
|\tilde\Psi(y\cdot z)-\tilde\Psi(y)| &\leq C|z|_{\G}\disp\sup_{|g|_{\G}\leq b|z|_{\G}}|X\tilde\Psi(y\cdot g)|\\
                                                &\leq CC_l|z|_{\G}\disp\sup_{|g|_{\G}\leq b|z|_{\G}}(1+|y\cdot g|_{\G}))^{-l},\label{2_21_10}
\end{align} 
where $C_l$ appears in \eqref{2_22_111}.
By the assumption $|z|_{\G}\leq D$ for some $D>0$, we have $|g|_{\G}\leq bD$. By taking $D>0$ such that $bD<1/2$, we have 
\begin{align}
1+|y\cdot g|_{\G}\geq 1+|y|_{\G}-|g|_{\G}\geq 1+|y|_{\G}-bD>\disp\frac{1}{2}+|y|_{\G}. 
\end{align}
Hence we obtain for any $l\in{\mathbb N}$, 
\begin{align}
(1+|y\cdot g|_{\G})^{-l}\leq C_l(1+|y|_{\G})^{-l}.\label{2_21_11}
\end{align}
By \eqref{2_21_10} and \eqref{2_21_11}, we have for $|z|_{\G}\leq D$,
\begin{align}
|\tilde\Psi(y\cdot z)-\tilde\Psi(y)| \leq C_l|z|_{\G}(1+|y|_{\G})^{-l}.\label{2_21_12}
\end{align}

Next we assume $j\neq 0$. Since $\delta_{2^j}\in {\mathrm {Aut}}(\G)$, by \eqref{2_21_12}, we have for $|\delta_{2^j}(z)|_{\G}\leq D$, 
\begin{align}
|\tilde\Psi_j(y\cdot z)-\tilde\Psi_j(y)| &= 2^{Nj}|\tilde\Psi(\delta_{2^j}(y\cdot z))-\tilde\Psi(\delta_{2^j}(y))|\\
                                                      &= 2^{Nj}|\tilde\Psi(\delta_{2^j}(y)\delta_{2^j}(z))-\tilde\Psi(\delta_{2^j}(y))|\\
                                                      &\leq C_l2^{Nj}|\delta_{2^j}(z)|_{\G}(1+|\delta_{2^j}(y)|_{\G})^{-l}. \label{2_21_13}
\end{align}
Since $|\delta_{2^j}(a)|_{\G}=2^j|a|_{\G}$ for $a\in\G$, by \eqref{2_21_13}, we can obtain the desired estimate.
\end{proof}

Next, we will show the following fractional chain rule on $\G$ (we refer to \cite{Chr}, \cite{Staff}, \cite{Tay} and so on).
\begin{prop}\label{3_11_1}
Assume that $F\in C^1$ with
\[F(0)=0,\ |F^{\prime}(x)|\leq C|x|^{\alpha-1},\ \alpha>1,\]
$0\leq\delta\leq 1$, $1<p,q<\infty$, $1<r\leq \infty$ and $p^{-1}=(\alpha-1)r^{-1}+q^{-1}$. If $u\in L^r(\G)\cap\dot{L}^q_{\delta}(\G)$,  then ${\mathcal L}^{\frac{\delta}{2}}F(u)\in L^p(\G)$ and 
\begin{align}
\|{\mathcal L}^{\frac{\delta}{2}}F(u)\|_p \leq C\|u\|_r^{\alpha-1}\|{\mathcal L}^{\frac{\delta}{2}}u\|_q.\label{aaaaaaa}
\end{align}
\end{prop}

\begin{proof}
When $\delta=1$, since $\|{\mathcal L}^{\frac{1}{2}}f\|_p\sim \sum_{j=1}^k\|X_jf\|_p$ (see \cite{Cou}, Appendix 3), we have
\begin{align}
\|{\mathcal L}^{\frac{1}{2}}F(u)\|_p &\sim \disp\sum_{j=1}^k\|X_jF(u)\|_p\\
                                                      &\leq \disp\sum_{j=1}^k\|F^{\prime}(u)\|_{r/(\alpha-1)}\|X_ju\|_q\\
                                                      &\leq C\|u\|_r^{\alpha-1}\|{\mathcal L}^{\frac{1}{2}}u\|_q. 
\end{align}
When $\delta=0$, by $F(0)=0$, we have
\begin{align}
|F(u)|=|F(u)-F(0)|\leq C|u|^{\alpha}.
\end{align}
So we obtain
\begin{align}
\|F(u)\|_p\leq C\||u|^{\alpha}\|_p\leq C\||u|^{\alpha-1}\|_{r/(\alpha-1)}\|u\|_q\leq C\|u\|_r^{\alpha-1}\|u\|_q.
\end{align}

Now we will prove \eqref{aaaaaaa} for $0<\delta<1$. Since $\int_{\G}\Psi_j(x)~dx=0$, we have
\begin{align}
\psi_j({\mathcal L})F(u)(x) &= \disp\int_{\G}F(u)(x^{\prime})\Psi_j({x^{\prime}}^{-1}\cdot x)~dx^{\prime}\\
                                         &= \disp\int_{\G}(F(u)(x^{\prime})-F(u)(x))\Psi_j({x^{\prime}}^{-1}\cdot x)~dx^{\prime}\\
                                         &= \disp\int_{\G}\left(\disp\int_0^1F^{\prime}(\theta u(x^{\prime})+(1-\theta)u(x))~d\theta\right)\times \\
                                         &\hspace{2.0cm}(u(x^{\prime})-u(x))\Psi_j({x^{\prime}}^{-1}\cdot x)~dx^{\prime}.\label{2_21_14}
\end{align}
Note that 
\begin{align}
\left|F^{\prime}(\theta u(x^{\prime})+(1-\theta)u(x))\right|
&\leq C|\theta u(x^{\prime})+(1-\theta)u(x)|^{\alpha-1}\\[5pt]
&\leq C(|u(x^{\prime})|^{\alpha-1}+|u(x)|^{\alpha-1})\\[5pt]
&= C(H(x^{\prime})+H(x)),
\end{align}
where we put $H(y)=|u(y)|^{\alpha-1}>0$. 
Since we can decompose
\[u=\disp\sum_{k=-\infty}^{\infty}\psi_k({\mathcal L})u=\disp\sum_{k=-\infty}^{\infty}\tilde\psi_k({\mathcal L})\psi_k({\mathcal L})u,\]
by \eqref{2_21_14}, we have
\begin{align}
&|\psi_j({\mathcal L})F(u)(x)| \\
&\leq CH(x)\disp\sum_{k=-\infty}^{\infty}\disp\int_{\G}|\tilde\psi_k({\mathcal L})\psi_k({\mathcal L})u(x^{\prime})-\tilde\psi_k({\mathcal L})\psi_k({\mathcal L})u(x)||\Psi_j({x^{\prime}}^{-1}\cdot x)|~dx^{\prime}\\
&+ \disp\sum_{k=-\infty}^{\infty}\disp\int_{\G}|\tilde\psi_k({\mathcal L})\psi_k({\mathcal L})u(x^{\prime})-\tilde\psi_k({\mathcal L})\psi_k({\mathcal L})u(x)||\Psi_j({x^{\prime}}^{-1}\cdot x)|H(x^{\prime})~dx^{\prime}.\\
\label{2263015}
\vspace{0.3cm}
\end{align}
\noindent

We assume that $G=1$ or $G=H$. Break the sum over $k$ into the case $k<j$ and $k\geq j$. Then we have
\begin{align}
&\disp\sum_{k=-\infty}^{\infty}\disp\int_{\G}|\tilde\psi_k({\mathcal L})\psi_k({\mathcal L})u(x^{\prime})-\tilde\psi_k({\mathcal L})\psi_k({\mathcal L})u(x)||\Psi_j({x^{\prime}}^{-1}\cdot x)|G(x')~dx^{\prime}\\
&= \disp\sum_{k<j}\disp\int_{\G}|\tilde\psi_k({\mathcal L})\psi_k({\mathcal L})u(x^{\prime})-\tilde\psi_k({\mathcal L})\psi_k({\mathcal L})u(x)||\Psi_j({x^{\prime}}^{-1}\cdot x)|G(x')~dx^{\prime}\\
&+\disp\sum_{k\geq j}\disp\int_{\G}|\tilde\psi_k({\mathcal L})\psi_k({\mathcal L})u(x^{\prime})-\tilde\psi_k({\mathcal L})\psi_k({\mathcal L})u(x)||\Psi_j({x^{\prime}}^{-1}\cdot x)|G(x')~dx^{\prime}\\
&= {\rm{I}}_{1,j}+{\rm{I}}_{2,j}.\label{226301}
\vspace{0.3cm}
\end{align}
\noindent
\underline{Estimate for ${\rm{I}}_{1,j}$}\\

For $k<j$, by Lemma \ref{2_21_2}, we have
\begin{align}
&\disp\int_{|x'^{-1}\cdot x|_{\G}\leq 2^{-k}}|\tilde\psi_k({\mathcal L})\psi_k({\mathcal L})u(x^{\prime})-\tilde\psi_k({\mathcal L})\psi_k({\mathcal L})u(x)||\Psi_j({x'}^{-1}\cdot x)|G(x')~dx'\\
&\leq C2^kM[\psi_k({\mathcal L})u](x)\disp\int_{|x'^{-1}\cdot x|_{\G}\leq 2^{-k}}|x'^{-1}\cdot x|_{\G}|\Psi_j(x'^{-1}\cdot x)|G(x')~dx'\\
&\leq C2^{k-j}M[\psi_k({\mathcal L})u](x)M[G](x),\label{226302}
\end{align}
where in the last inequality, we used the fact that 
\begin{align}
\disp\int_{|x'^{-1}\cdot x|_{\G}\leq 2^{-k}}|x'^{-1}\cdot x|_{\G}|\Psi_j(x'^{-1}\cdot x)|~dx' \leq 2^{-j}(G*\eta_{j,N+2})(x) \label{bbbbbbb}
\end{align}  
and \eqref{2_21_3}. The estimate \eqref{bbbbbbb} follows from \eqref{2_22_1} with $l=N+3$.
On the other hand, we have
\begin{align}
&\disp\int_{|x'^{-1}\cdot x|_{\G}\geq 2^{-k}}|\tilde\psi_k({\mathcal L})\psi_k({\mathcal L})u(x^{\prime})-\tilde\psi_k({\mathcal L})\psi_k({\mathcal L})u(x)||\Psi_j({x'}^{-1}\cdot x)|G(x')~dx'\\
&\leq C\disp\int_{|x'^{-1}\cdot x|_{\G}\geq 2^{-k}}\left\{M[\psi_k({\mathcal L})u](x^{\prime})+M[\psi_k({\mathcal L})u](x)\right\}|\Psi_j({x'}^{-1}\cdot x)|G(x')~dx'\\
&\leq C\disp\int_{|x'^{-1}\cdot x|_{\G}\geq 2^{-k}}M[\psi_k({\mathcal L})u](x')2^{Nj}(1+2^j|x'^{-1}\cdot x|_{\G})^{-(N+3)}G(x')~dx'\\
&+CM[\psi_k({\mathcal L})u](x)\disp\int_{|x'^{-1}\cdot x|_{\G}\geq 2^{-k}}2^{Nj}(1+2^j|x'^{-1}\cdot x|_{\G})^{-(N+3)}G(x')~dx^{\prime}.\\\label{226303}
\end{align}
Since $|x'^{-1}\cdot x|_{\G}\geq 2^{-k}$, we have 
\[(1+2^j|x'^{-1}\cdot x|_{\G})^{-1}\leq 2^{-j}|x'^{-1}\cdot x|_{\G}^{-1}\leq 2^{k-j}.\]
Thus, we obtain
\begin{align}
\eqref{226303} &\leq C2^{k-j}M[M[\psi_k({\mathcal L})u]G](x)+C2^{k-j}M[\psi_k({\mathcal L})u](x)M[G](x),\\\label{226304}
\end{align}
where we used \eqref{2_21_3}. 
So, by \eqref{226302} and \eqref{226304}, we have
\begin{align}
{\rm{I}}_{1,j} &\leq C\disp\sum_{k<j}2^{k-j}\{M[M[\psi_k({\mathcal L})u]G](x)+M[\psi_k({\mathcal L})u](x)M[G](x)\}.\\\label{226305}
\vspace{0.3cm}
\end{align}
\noindent
\underline{Estimate for ${\rm{I}}_{2,j}$}\\

For $k\geq j$, we have
\begin{align}
&\disp\int_{\G}|\tilde\psi_k({\mathcal L})\psi_k({\mathcal L})u(x^{\prime})-\tilde\psi_k({\mathcal L})\psi_k({\mathcal L})u(x)||\Psi_j({x'}^{-1}\cdot x)|G(x')~dx'\\
&\leq \disp\int_{\G}\{M[\psi_k({\mathcal L})u](x')+M[\psi_k({\mathcal L})u](x)\}2^{Nj}(1+2^j|x'^{-1}\cdot x|_{\G})^{-(N+2)}G(x')dx^{\prime}\\[5pt]
&\leq C_N\{M[M[\psi_k({\mathcal L})u]G](x)+M[\psi_k({\mathcal L})u](x)M[G](x)\},
\end{align}
by the same argument of \eqref{226304}. Hence we have
\begin{align}
{\rm{I}}_{2,j} &\leq C\disp\sum_{k\geq j}\{M[M[\psi_k({\mathcal L})u]G](x)+M[\psi_k({\mathcal L})u](x)M[G](x)\}.\label{226306}
\end{align}
Hence by \eqref{2263015}, \eqref{226301}, \eqref{226305} and \eqref{226306}, we obtain
\begin{align}
|\psi_j({\mathcal L})F(u)(x))| &\lesssim H(x)\left[\disp\sum_{k<j}2^{k-j}\{M[\psi_k({\mathcal L})u](x)+M[M[\psi_k({\mathcal L})u]](x)\}\right.\\
             &\hspace{2.0cm}+ \left.\disp\sum_{k\geq j}\{M[\psi_k({\mathcal L})u](x)+M[M[\psi_k({\mathcal L})u]](x)\}\right]\\
             &+\left[\disp\sum_{k<j}2^{k-j}\{M[\psi_k({\mathcal L})u](x)M[H](x)+M[M[\psi_k({\mathcal L})u]H](x)\}\right.\\
             &\hspace{2.0cm}+ \left.\disp\sum_{k\geq j}\{M[\psi_k({\mathcal L})u](x)M[H](x)+M[M[\psi_k({\mathcal L})u]H](x)\}\right].\label{2263016}
\end{align}
 Now we use the following lemma (for instance, see \cite{Tay}, Lemma 4.2.).
\begin{lem}\label{2263017}
If $\delta<1$, then 
\[\disp\sum_{j=-\infty}^{\infty}2^{2j\delta}\left|\disp\sum_{k<j}2^{k-j}a_k\right|^2\leq C\disp\sum_{k=-\infty}^{\infty}2^{2k\delta}|a_k|^2.\]
If $\delta>0$, then
\[\disp\sum_{j=-\infty}^{\infty}2^{2j\delta}\left|\disp\sum_{k\geq j}a_k\right|^2\leq C\disp\sum_{k=-\infty}^{\infty}2^{2k\delta}|a_k|^2.\]
\end{lem}
By \eqref{2263016} and Lemma \ref{2263017}, for $0<\delta<1$, we have
\begin{align}
&\disp\sum_{j=-\infty}^{\infty}2^{2j\delta}|\psi_j({\mathcal L})F(u)(x)|^2\\
&\lesssim H^2(x)\disp\sum_{j=-\infty}^{\infty}2^{2j\delta}\left|\disp\sum_{k<j}2^{k-j}\{M[\psi_k({\mathcal L})u](x)+M[M[\psi_k({\mathcal L})u]](x)\}\right|^2\\
&+H^2(x)\disp\sum_{j=-\infty}^{\infty}2^{2j\delta}\left| \disp\sum_{k\geq j}\{M[\psi_k({\mathcal L})u](x)+M[M[\psi_k({\mathcal L})u]](x)\}\right|^2\\
&+ \disp\sum_{j=-\infty}^{\infty}2^{2j\delta}\left|\disp\sum_{k<j}2^{k-j}\{M[\psi_k({\mathcal L})u](x)M[H](x)+M[M[\psi_k({\mathcal L})u]H](x)\}\right|^2\\
&+\disp\sum_{j=-\infty}^{\infty}2^{2j\delta}\left|\disp\sum_{k\geq j}\{M[\psi_k({\mathcal L})u](x)M[H](x)+M[M[\psi_k({\mathcal L})u]H](x)\}\right|^2\\
&\lesssim H^2(x)\disp\sum_{k=-\infty}^{\infty}2^{2k\delta}\{M[\psi_k({\mathcal L})u](x)+M[M[\psi_k({\mathcal L})u]](x)\}^2\\
&+\disp\sum_{k=-\infty}^{\infty}2^{2k\delta}\{M[\psi_k({\mathcal L})u](x)M[H](x)+M[M[\psi_k({\mathcal L})u]H](x)\}^2.\label{2263018}
\end{align}
So, by \eqref{2263018}, we have
\begin{align}
\|{\mathcal L}^{\frac{\delta}{2}}F(u)\|_p &\lesssim \left\|\left(\disp\sum_{j=-\infty}^{\infty}2^{2j\delta}\left|\psi_j({\mathcal L})F(u)\right|^2\right)^\frac{1}{2}\right\|_p\\
&\lesssim \left\|H\left(\disp\sum_{k=-\infty}^{\infty}2^{2k\delta}\{M[\psi_k({\mathcal L})u]+M[M[\psi_k({\mathcal L})u]]\}^2\right)^{\frac{1}{2}}\right\|_p\\
&+\disp\left\|\left(\disp\sum_{k=-\infty}^{\infty}2^{2k\delta}\{M[\psi_k({\mathcal L})u]M[H]+M[M[\psi_k({\mathcal L})u]H]\}^2\right)^{\frac{1}{2}}\right\|_p\\
&=\rm{II}+\rm{III}.\label{2263019}
\vspace{0.5cm}
\end{align}
\noindent
\underline{Estimate for $\rm{II}$}

By H$\rm\ddot{o}$lder's inequality and Proposition \ref{2_21_21} (the vector-valued maximal theorem of Fefferman - Stein), we have
\begin{align}
&\left\|H\left(\disp\sum_{k=-\infty}^{\infty}2^{2k\delta}\{M[\psi_k({\mathcal L})u]+M[M[\psi_k({\mathcal L})u]]\}^2\right)^{\frac{1}{2}}\right\|_p\\
&\lesssim \|H\|_{\frac{r}{\alpha-1}}\left\|\left(\sum_{k=-\infty}^{\infty}2^{2k\delta}\{M[\psi_k({\mathcal L})u]+M[M[\psi_k({\mathcal L})u]]\}^2\right)^{\frac{1}{2}}\right\|_q\\
&\lesssim \|H\|_{\frac{r}{\alpha-1}}\left(\left\|\left(\disp\sum_{k=-\infty}^{\infty}2^{2k\delta}\{M[\psi_k({\mathcal L})u]\}^2\right)^{\frac{1}{2}}\right\|_q+\left\|\left(\disp\sum_{k=-\infty}^{\infty}2^{2k\delta}\{M[M[\psi_k({\mathcal L})u]]\}^2\right)^{\frac{1}{2}}\right\|_q\right)\\
&\lesssim \|H\|_{\frac{r}{\alpha-1}}\left(\left\|\left(\disp\sum_{k=-\infty}^{\infty}2^{2k\delta}|\psi_k({\mathcal L})u|^2\right)^{\frac{1}{2}}\right\|_q+\left\|\left(\disp\sum_{k=-\infty}^{\infty}2^{2k\delta}|\psi_k({\mathcal L})u|^2\right)^{\frac{1}{2}}\right\|_q\right)\\
&\lesssim \|H\|_{\frac{r}{\alpha-1}}\|{\mathcal L}^{\frac{\delta}{2}}u\|_q\label{2263020}
\end{align}
for $1/p=(\alpha-1)/r+1/q$. 
\vspace{0.5cm}\\
\noindent
\underline{Estimate for $\rm{III}$}\\

We have
\begin{align}
&\disp\left\|\left(\disp\sum_{k=-\infty}^{\infty}2^{2k\delta}\{M[\psi_k({\mathcal L})u]M[H]+M[M[\psi_k({\mathcal L})u]H]\}^2\right)^{\frac{1}{2}}\right\|_p\\
&\leq \disp\left\|\left(\disp\sum_{k=-\infty}^{\infty}2^{2k\delta}\{M[\psi_k({\mathcal L})u]M[H]\}^2\right)^{\frac{1}{2}}\right\|_p\\
&+\disp\left\|\left(\disp\sum_{k=-\infty}^{\infty}2^{2k\delta}\{M[M[\psi_k({\mathcal L})u]H]\}^2\right)^{\frac{1}{2}}\right\|_p\\
&=\rm{III}_1+\rm{III}_2.\label{2263021}
\end{align}
We will estimate $\rm{III}_1$. By H$\rm\ddot{o}$lder's inequality, Proposition \ref{2_21_20} ($L^p$ boundness of $M$) and Proposition \ref{2_21_21} (the vector-valued maximal theorem of Fefferman - Stein), we obtain
\begin{align}
&\disp\left\|\left(\disp\sum_{k=-\infty}^{\infty}2^{2k\delta}\{M[\psi_k({\mathcal L})u]M[H]\}^2\right)^{\frac{1}{2}}\right\|_p\\
&\lesssim \|M[H]\|_{\frac{r}{\alpha-1}}\disp\left\|\left(\disp\sum_{k=-\infty}^{\infty}2^{2k\delta}\{M[\psi_k({\mathcal L})u]\}^2\right)^{\frac{1}{2}}\right\|_p\\
&\lesssim \|H\|_{\frac{r}{\alpha-1}}\disp\left\|\left(\disp\sum_{k=-\infty}^{\infty}2^{2k\delta}|\psi_k({\mathcal L})u|^2\right)^{\frac{1}{2}}\right\|_p\\
&\lesssim \|H\|_{\frac{r}{\alpha-1}}\|{\mathcal L}^{\frac{\delta}{2}}u\|_q\label{2263022}
\end{align}
for $1/p=(\alpha-1)/r+1/q$. Similarly, we will estimate $\rm{III}_2$. By H$\rm\ddot{o}$lder's inequality and Proposition \ref{2_21_21} (the vector-valued maximal theorem of Fefferman - Stein), we obtain
\begin{align}
&\disp\left\|\left(\disp\sum_{k=-\infty}^{\infty}2^{2k\delta}\{M[M[\psi_k({\mathcal L})u]H]\}^2\right)^{\frac{1}{2}}\right\|_p\\
&\lesssim \left\|H\left(\disp\sum_{k=-\infty}^{\infty}2^{2k\delta}\left\{M[\psi_k({\mathcal L})u]\right\}^2\right)^{\frac{1}{2}}\right\|_p\\
&\lesssim \|H\|_{\frac{r}{\alpha-1}}\left\|\left(\disp\sum_{k=-\infty}^{\infty}2^{2k\delta}|\psi_k({\mathcal L})u|^2\right)^{\frac{1}{2}}\right\|_q\\
&\lesssim \|H\|_{\frac{r}{\alpha-1}}\|{\mathcal L}^{\frac{\delta}{2}}u\|_q\label{2263023}
\end{align}
for $1/p=(\alpha-1)/r+1/q$. Therefore by \eqref{2263019}, \eqref{2263020}, \eqref{2263022} and \eqref{2263023}, we have
\begin{align}
\|{\mathcal L}^{\frac{\delta}{2}}F(u)\|_p&\lesssim \|H\|_{\frac{r}{\alpha-1}}\|{\mathcal L}^{\frac{\delta}{2}}u\|_q\\
                                                                       &=\|u\|_r^{\alpha-1}\|{\mathcal L}^{\frac{\delta}{2}}u\|_q.
\end{align}
This completes the proof of Proposition \ref{3_11_1}.
\end{proof}

By Proposition \ref{3_11_1}, we address the proof of  the generalized fractional chain rule on $\G$. 

\noindent
{\bf{The proof of Lemma \ref{3_11_2}}}. When $l=1$, since \eqref{3_11_3} means the fractional chain rule for $0\leq \delta\leq 1$, it has been obtained in Proposition \ref{3_11_1}. The general case follows by induction on $l$. Assume that \eqref{3_11_3} has been proved for $0\leq \delta\leq l-1$, $l\geq 2$. Then by Proposition \ref{3_11_4}, we have  
\begin{align}
\|{\mathcal L}^{\frac{\delta+1}{2}}F(u)\|_p &\leq \disp\sum_{m=1}^k\|{\mathcal L}^{\frac{\delta}{2}}X_mF(u)\|_p\\
                                                                 &=\disp\sum_{m=1}^k\|{\mathcal L}^{\frac{\delta}{2}}(F^{\prime}(u)X_mu)\|_p\\
                                                                 &\leq \disp\sum_{m=1}^k(\|F^{\prime}(u)\|_{r/(\alpha-1)}\|{\mathcal L}^{\frac{\delta}{2}}X_mu\|_q+\|{\mathcal L}^{\frac{\delta}{2}}F^{\prime}(u)\|_{\tau}\|X_mu\|_\sigma ),\\
                                                                 &= \disp\sum_{j=1}^k({\rm{I}}_{1,m}+{\rm{I}}_{2,m})\label{3_12_4}
\end{align}
where $1/p=(\alpha-1)/r+1/q$ and $1/p=1/\tau+1/\sigma$. In the first inequality above, we use the fact, $\|{\mathcal L}^{\frac{\delta+1}{2}}F(u)\|_p\leq C\sum_{m=1}^k\|{\mathcal L}^{\frac{\delta}{2}}X_mF(u)\|_p$ (in detail, see \cite{Cou}, Proposition 21). Since $|F^{\prime}(u)|\leq |u|^{\alpha-1}$, for the first term $\sum_{m=1}^k{\rm{I}}_{m,j}$, we have
\begin{align}
\disp\sum_{m=1}^k{\rm{I}}_{1,m} &= \disp\sum_{m=1}^k\|F^{\prime}(u)\|_{r/(\alpha-1)}\|{\mathcal L}^{\frac{\delta}{2}}X_mu\|_q\\
           &\leq C\disp\|u\|^{\alpha-1}_r\sum_{m=1}^k\|X_mu\|_{\dot{L}^q_{\delta}}\\
           &\leq C\|u\|^{\alpha-1}_r\|u\|_{\dot{L}^q_{\delta+1}}\\
           &=C\|u\|_r^{\alpha-1}\|{\mathcal L}^{\frac{\delta+1}{2}}u\|_q. \label{3_12_1}
\end{align}  
In the last inequality above, we use the fact, $\|f\|_{\dot{L}_{s}^q}\sim \sum_{m=1}^k\|X_mf\|_{\dot{L}^q_{s-1}}$ (for details, see \cite{Hu}, Corollary 21). On the other hand, for the second term $\disp\sum_{m=1}^k{\rm{I}}_{2,m}$, since we assume that \eqref{3_11_3} has been proved for $0\leq \delta\leq l-1$,  we obtain 
\begin{align}
\disp\sum_{m=1}^k{\rm{I}}_{2,m}=\|{\mathcal L}^{\frac{\delta}{2}}F^{\prime}(u)\|_{\tau}\disp\sum_{m=1}^k\|X_mu\|_{\sigma}\leq C\|u\|^{\alpha-2}_r\|{\mathcal L}^{\frac{\delta}{2}}u\|_\rho\disp\sum_{m=1}^{k}\|X_mu\|_{\sigma},\\\label{3_12_3}
\end{align} 
where $1/\rho=1/\tau-(\alpha-2)/r=1/q+1/r-1/\sigma$. Since $\rho$ and $\sigma$ satisfy $1/\rho+1/\sigma=1/q+1/r$ for $1<q,\rho,\sigma<\infty$ and $1<r\leq \infty$, we set
\[\disp\frac{1}{\rho}=\disp\frac{\delta}{(\delta+1)q}+\disp\frac{1}{(\delta+1)r},\ \disp\frac{1}{\sigma}=\disp\frac{1}{(\delta+1)q}+\disp\frac{\delta}{(\delta+1)r}.\]
By Proposition \ref{3_11_5} (the Gagliardo - Nirenberg inequality), we have
\[\|{\mathcal L}^{\frac{\delta}{2}}u\|_{\rho}\leq C\|{\mathcal L}^{\frac{\delta+1}{2}}u\|^{\frac{\delta}{\delta+1}}_q\|u\|_r^{\frac{1}{\delta+1}}\]
and
\[\|X_mu\|_{\sigma}\leq C\|{\mathcal L}^{\frac{1}{2}}u\|_\sigma \leq C\|{\mathcal L}^{\frac{\delta+1}{2}}u\|_q^{\frac{1}{\delta+1}}\|u\|_r^{\frac{\delta}{\delta+1}}.\]
So by \eqref{3_12_3}, we have
\begin{align}
\eqref{3_12_3} &\leq C\|u\|_r^{\alpha-1}\|{\mathcal L}^{\frac{\delta+1}{2}}u\|_q.\label{3_12_5}
\end{align}
By \eqref{3_12_4}, \eqref{3_12_1} and \eqref{3_12_5}, we have for $1\leq \delta+1\leq l$, 
\[\|{\mathcal L}^{\frac{\delta+1}{2}}F(u)\|_p\leq C\|u\|_r^{\alpha-1}\|{\mathcal L}^{\frac{\delta+1}{2}}u\|_q.\]
This completes the proof of Lemma \ref{3_11_2}.\\

\section*{Acknowledgements}
This work was supported by
JSPS KAKENHI Grant Number JP 21K03333.

\noindent
Hiroyuki HIRAYAMA\\
Faculty of Education, University of Miyazaki,\\
1-1, Gakuenkibanadai-nishi, Miyazaki, 889-2192 Japan\\
e-mail: h.hirayama@cc.miyazaki-u.ac.jp\\
\\
\noindent
Yasuyuki OKA\\
School of Liberal Arts and Sciences,  Daido university, \\
10-3 Takiharu-cho, Minami-ku, Nagoya 457-8530 Japan \\
e-mail: y-oka@daido-it.ac.jp

\end{document}